\documentclass{amsart}

\usepackage{graphicx}
\usepackage{amsmath,amssymb,amsthm,amsgen,mathrsfs}
\usepackage{color}
\usepackage[top=2cm,bottom=2cm,left=3cm,right=3cm]{geometry}
\usepackage{enumitem}
\usepackage{environments}
\usepackage{cite}

\author{Thomas Hudson}
\title{An existence result for Discrete Dislocation Dynamics in three dimensions}

\newcommand*{\smfrac}[2]{{\textstyle \frac{#1}{#2}}}

\def\AXint#1#2#3{{\setbox0=\hbox{$#1{#2#3}{\int}$}
\vcenter{\hbox{$#2#3$}}\kern-.5\wd0}}

\def\<{\langle}
\def\>{\rangle}

\def\R{\mathbb{R}}

\def\N{\mathbb{N}}
\def\Z{\mathbb{Z}}

\def\Sph{\mathbb{S}}

\def\CC{{\rm C}}
\def\HH{{\rm H}}

\def\dd{{\rm d}}

\def\dx{\,\dd x}
\def\dt{\,\dd t}

\def\to{\rightarrow}

\def\sep{\,|\,}
\def\bsep{\,\b|\,}
\def\Bsep{\,\B|\,}

\newcommand{\Rstr}{%
  \hspace{.25ex}\raisebox{-.3ex}{\scalebox{1.5}{\reflectbox{\rotatebox[origin=br]{-90}{$\lnot$}}}\hspace{-.3ex}}\,%
}

\DeclareMathOperator*{\argmin}{\mathrm{argmin}}

\def\BLO{\mathcal{L}}

\def\eps{\varepsilon}
\def\del{\delta}

\def\b{\big}
\def\B{\Big}
\def\bg{\bigg}
\def\Bg{\Bigg}

\def\Mass{\mathbf{M}}
\def\MRatio{\boldsymbol{\Theta}}
\def\Ione{\mathscr{I}_1}
\def\Itwo{\mathscr{I}_2}

\def\Id{\mathsf{I}}
\def\Alt{\mathsf{A}}
\def\Dsf{\mathsf{D}}
\def\GMat{\mathsf{G}}
\def\Jker{\mathsf{J}^\eps}
\def\Kker{\mathsf{K}^\eps}
\def\Elas{\mathsf{C}}
\def\Proj{\mathsf{P}}

\def\id{\mathrm{id}}

\def\PK{f^{\mathrm{PK}}}

\def\Adm{\mathscr{A}}

\def\Haus{\mathcal{H}}

\def\div{\mathrm{div}}

\def\supp{\mathrm{supp}}

\def\E{\mathcal{E}}
\def\I{\mathcal{I}}

\def\Latt{\mathscr{L}}

\begin{document}
\begin{abstract}
  We present a mathematical framework within which Discrete Dislocation Dynamics in three
  dimensions is well--posed. By considering smooth distributions of slip, we derive a
  regularised energy for curved dislocations, and rigorously
  derive the Peach--Koehler force on the dislocation network via an inner variation.
  We propose a dissipative evolution law which is cast as a generalised gradient flow,
  and using a discrete--in--time approximation scheme,
  existence and regularity results are obtained for the evolution, up until the first time at
  which an infinite density of dislocation lines forms.
\end{abstract}

\maketitle


\section{Introduction}\label{sec:introduction}
In crystalline materials, plastic behaviour is characterised by the generation of slip,
which is the process by which the planes of the material's lattice structure are reordered.
As the material deforms, slip is propagated via the motion of dislocations, which are
topological line defects found in regions where the lattice mismatch required for slip to occur is
most concentrated \cite{HL82,HB11}.
From the very beginnings of the study of dislocations \cite{Orowan1934,Polanyi1934,T34},
linear continuum theories have been used to
model these defects with great success \cite{V07,B39a,B39b,PK50,Blin55,M63,BBS80}, despite the
fact that unphysical singularities are induced in the stress, strain
and energy density fields at the dislocation lines.
These singularities are a signature of the breakdown of
the assumption that the material behaves as a continuum close to dislocation lines, and 
although no continuum theory is able to accurately capture the properties of dislocation cores,
continuum approaches have nevertheless been highly successful at capturing bulk behaviour.
Indeed, in a
series of recent mathematical works it has been rigorously demonstrated
that linear elastoplasticity theory provides an excellent
prediction of the strain caused by defects on the atomistic scale \cite{HO14,HO15,EOS16,H17,BBO17}.

Since dislocation motion determines the plastic behaviour of crystalline materials, one of the
principal aims of studying these objects is to understand the physical laws which govern their
microscopic motion, and consequently, to obtain an accurate description of the evolution of
crystal plasticity on a macroscopic scale. To that end, two broad approaches to
modelling dislocation motion have developed. \emph{Phase field} models consider
a continuous distribution of dislocations, an approach which has its roots in the
classic \emph{Peierls--Nabarro model} \cite{P40,N52}; modern examples of
such modelling approaches include \cite{GB99,KCO02,ZY15}. A variety of 
mathematical results concerning models in this class have been obtained, including well--posedness
\cite{BCLM08,CHMR10}, long--time asymptotics \cite{PV15,PV16,PV17}, and homogenization
results \cite{GM05,GM06,CGM11,MP12a,MP12b}. While these models have good mathematical structure,
they are typically limited to considering only one family of slip planes at once; moreover,
simulating phase field models on a microscopic scale is computationally
intensive, since a high resolution mesh is required to accurately resolve the level sets
corresponding to individual dislocations.

The second approach is \emph{Discrete Dislocation Dynamics} (DDD), in which
dislocations lines are described as curves within the crystal, and are driven by the action of
the \emph{Peach--Koehler force} \cite{PK50}.
The fact that the dislocations alone are tracked in this approach has the advantage of
drastically reducing the computational complexity in comparison with phase field approaches, and
as such, DDD has been used as a simulation technique for studying plasticity since the early 1990s
\cite{GSLL89,AG90a,AG90b,VdGN95,BC06,ACetal07,DMMQGK11}.
While a significant mathematical literature has developed which considers
one-- and two--dimensional DDD models for the motion of straight dislocations
\cite{CG99,CL05,ADLGP14,vMM14,BFLM15,ADLGP17,BvMM17,BM17,H17,HM17}, few mathematical results
concerning DDD in a three--dimensional setting exist to date, and in part, this appears to be due
to the lack of a clear mathematical statement of what the evolution problem for DDD should be in
this setting.

This paper therefore seeks to bring together many of the
ideas already present in the literature in laying out a well--posed mathematical formulation of DDD
which is both general enough to encompass evolution problems similar to those considered by
Materials Scientists and Engineers simulating DDD in practice, and mathematically concrete enough
to allow the development of further results concerning dislocations in three dimensions.
In particular, we hope to open paths towards a deeper mathematical understanding of
the numerical schemes used to simulate DDD in practice.

\subsection{A regularised theory of DDD in three dimensions}
\label{sec:regul-theory-ddd}
As mentioned above, we seek to develop a well--posed mathematical theory of DDD in three
dimensions. In order to be physically--relevant, practical for computation and amenable to
mathematical analysis, we make several requirements of this theory:
\begin{enumerate}
\item Dislocations should be curved and satisfy the physically--necessary condition that they
  are the boundaries of regions of slip.
\item The stress, strain and internal energy density induced in the material by the presence of
  dislocations are required to be non--singular.
\item The energy of and Peach--Koehler force on a configuration of dislocations should take on
  explicit expressions in terms of an integral kernel which are computable with a quantifiable
  error.
\item The underlying material is assumed to be linearly elastic, but need not be isotropic.
\item Dislocation motion should dissipate internal energy.
\end{enumerate}
The first of these conditions is a kinematic requirement for a theory of dislocations to make
sense. In common with several recent mathematical works \cite{CGM15,CGO15,SvG16a,SvG16b},
this condition is encoded by describing dislocations as \emph{closed 1--currents}, which may be
viewed closed oriented Lipschitz curves satisfying certain topological constraints.

In order to satisfy the second condition, we construct a regularised version
of the classical linear theory via a similar approach to that used in \cite{CAWB06,CGO15}:
Starting from a regularised distribution of slip and using the ideas of Mura \cite{M91},
we derive an expression for the internal energy as a double integral which depends on the boundary
of the slip surfaces alone. As this expression depends only on the boundaries
of the surfaces over which slip has occurred, which correspond exactly to the dislocations in the
material, this energy is furthermore consistent with the first requirement made above.
Using the Fourier analytic ideas of \cite{BBS80}, we study the energy, providing a computable
expression for the integral kernel without requiring an
explicit expression of the elastic Green's function, which renders the theory broad enough to
satisfy both the third and fourth conditions.
By performing an inner variation of the energy with respect to the positions of dislocations,
we are able to rigorously derive an expression for the Peach--Koehler force.

To fulfil the final condition, we choose to formulate an evolution law for DDD as a generalised
gradient flow \cite{AGS} by prescribing a dissipation potential expressed in terms of the velocity
field perpendicular to the dislocation line.
This framework enables us to prove our main result, that the evolution problem is well--posed.

\subsection{Outline}

\S\ref{sec:notation-conventions} provides a record of the notation used throughout the paper for
the reader's convenience, and \S\ref{sec:main-results} is devoted to an exposition of the main
results of the paper, which are Theorem~\ref{th:Energy}, providing the properties of the regularised energy; Theorem~\ref{th:Force}, which provides properties of the
configurational force on dislocations; and Theorem~\ref{th:WellPosedness}, which asserts
the well--posedness of the model for DDD considered here.

The subsequent sections of the paper are then devoted to proving these results.
\S\ref{sec:elast-energy-disl} describes the regularisation procedure
applied to dislocations and derivation of the explicit representation of the energy due to
dislocations; \S\ref{sec:deform-disl-peach} derives the Peach--Koehler
(or configurational) force of a dislocation by inner variation; and 
\S\ref{sec:evol-probl-exist} prove existence, uniqueness, and regularity results for the evolution.

As mentioned above, we consider configurations of slip and dislocations as
integral currents \cite{deRham,FF60,Federer,Morgan}, since these are the correct
mathematical objects to describe the topological restrictions on dislocations
\cite{AO05,CGM15,CGO15,SvG16a,SvG16b}. Currents generalise the notion of
distributions \cite{Freidlander} to a geometric setting, and while the theory of these objects
can be forbidding, in our setting, the reader should always have in mind surfaces and curves.
For convenience, Appendix~\ref{appendix} recalls the definitions and basic theory related to these
objects that is used here.

\subsection{Notation}
\label{sec:notation-conventions}
The following notational conventions will be used throughout
the paper.

\subsubsection{Tensors}
\begin{itemize}
\item Important tensors which are fixed throughout (rather than variables) are generally denoted using sans
  serif fonts, e.g. $\Kker$.
\item Subscript indices always refer to components in Cartesian coordinates, e.g. $f_{abc}$.
\item Subscript indices appearing after a comma denote partial derivatives: e.g.
  $f_{i,j} = \smfrac{\partial f_i}{\partial x_j}$.
\item The Einstein summation convention is used throughout, so repeated indices within an expression are always
  summed, e.g. $a_{ijik}b_k = \sum_{i,k=1}^3 a_{ijik}b_k$.
\item Tensor products of vectors are denoted $a\otimes b$.
\item $\wedge$ denotes the usual alternating product acting on vectors and covectors (for further
  details, see \S\ref{sec:vectors-covectors}).
\item The dot products between vectors in $\R^3$ is denoted $v\cdot\tau$.
\item $\Id\in\R^{3\times3}$ always denotes the identity matrix.
\item $\Alt\in\R^{3\times3\times3}$ denotes the alternating tensor, which satisfies
  \begin{equation*}
    \Alt_{ijk} = \begin{cases}
      +1 & (ijk)\text{ is an even permutation of }(123)\\
      -1 & (ijk)\text{ is an odd permutation of }(123),\\
      0  & \text{otherwise.}
    \end{cases}
  \end{equation*}
\item $\Elas\in\R^{3\times3\times3\times3}$ denotes an elasticity tensor, which satisfies the
  \emph{major symmetry} $\Elas_{abcd}=\Elas_{cdab}$, the \emph{minor symmetries}
  $\Elas_{abcd} = \Elas_{bacd}=\Elas_{abdc}$, and a \emph{Legendre--Hadamard condition}, i.e. there exists
  $c_0>0$ such that
  \begin{equation*}
     \Elas_{abcd}v_ak_bv_ck_d \geq c_0 |v|^2|k|^2\quad\text{for all }v,k\in\R^3.
  \end{equation*}
\item $\GMat:\R^3\to\R^{3\times3}$ denotes the \emph{elastic Green's function}, i.e. the fundamental
  solution of the differential operator $-\Elas_{abcd}u_{c,db}$, which solves
  \begin{equation}\label{eq:GMatEqn}
    -\Elas_{\alpha jkl}\GMat_{\beta k,lj} = \Id_{\alpha\beta}\delta_0
  \end{equation}
  in the sense of distributions.
\item $\GMat^\eps:=\GMat*\varphi^\eps$ denotes a regularised version of the elastic Green's function,
  solving
  \begin{equation}\label{eq:GMatEpsEqn}
    -\Elas_{\alpha jkl}\GMat^\eps_{\beta k,lj} = \Id_{\alpha\beta}\varphi^\eps
  \end{equation}
  in the sense of distributions, where $\varphi^\eps$ is a smooth, positive, radially symmetric
  function satisfying $\int_{\R^3}\varphi^\eps\dx = 1$.
\end{itemize}

\subsubsection{Sets, currents, measures and integration}
\begin{itemize}
\item $\overline{B_r(s)}$ denotes the closed ball of radius $r$, centred at $s\in\R^3$.
\item $\Sigma$ will denote a $2$--rectifiable subset of $\R^3$ (i.e. a
  generalised surface) and $\Gamma$ will denote a $1$--rectifiable subset of $\R^3$
  (i.e. a generalised curve).
\item $T$ and $S$ will denote currents, and $\partial$ is the usual boundary
  operator.
\item $\Ione(\R^3;\Latt)$ and $\Itwo(\R^3;\Latt)$ denote space of integral 1-- and 2--currents
  with multiplicities in a given lattice $\Latt\subset\R^3$.
\item $S\Rstr A$ denotes the restriction of a current $S$ to $A$.
\item $F_\#S$ denotes the pushforward of the current $S$ by $F$, i.e. the current corresponding
  to the image $F(S)$.
\item $\Mass(S)$ denotes the mass of a current $S$, defined in \eqref{eq:Mass}.
\item $\MRatio(S)$ denotes the maximal mass ratio of a current $S$, defined in \eqref{eq:MassRatio}.
\item $\Haus^m$ denotes the $m$--dimensional Hausdorff measure on $\R^3$.
\item $\int_A f(t) \dd\mu(t)$ denotes the Lebesgue integral of a Borel measurable function
  $f$ with respect to the measure $\mu$ restricted to a Borel--measurable set $A$.
\end{itemize}
For some additional details on the basic theory of currents, see Appendix~\ref{appendix}.

\subsubsection{Functions and function spaces}  
\begin{itemize}
\item Lebesgue, Sobolev and H\"older spaces are all given standard
  notation, i.e. $L^p$, $H^k$, $C^k$, as are their norms.
\item The $C^{0,\gamma}$ H\"older seminorm is denoted $[F]_\gamma$.
\item The space of smooth $m$--forms is denoted $\mathscr{D}^m(\R^3)$ (see \ref{sec:forms} for a
  full definition).
\item We set $\dot{H}^1(\R^3) := \{\nabla u\in L^2(\R^3)\}$, which is equipped with the seminorm
  $u\mapsto \|\nabla u\|_{L^2(\R^3)}$.
\item The space of bounded linear operators mapping a Banach space $X$ to a Banach space $Y$ is
  denotes $\BLO(X,Y)$.
\item The $m$th Frechet derivative of a function $f$ at a point $x$ is denoted $D^mf$, and its (multilinear)
  action on vector $v_1,\ldots,v_m\in\R^3$ is denoted $D^mf(x)[v_1,\ldots,v_m]$.
\item The pullback of a function $G$ by $F$ is denoted $F^\#G$ (see
  \S\ref{sec:pushforward-pullback} for a full definition).
\item The identity mapping on $\R^3$ is denoted $\id$.
\item The Fourier transform is of a function $f$ is denoted $\widehat{f}$, and (for $f\in L^1(\R^3)$)
  we use the definition
  \begin{equation*}
  \widehat{f}(k) := \int_{\R^3}f(x)\mathrm{e}^{-i(k,x)}\dx,\quad\text{so that}
  \quad f(x) = \frac{1}{(2\pi)^3}\int_{\R^3}\widehat{f}(k)\mathrm{e}^{i(k,x)}\dd k.
\end{equation*}
\item For numbers, a bar denotes complex conjugation, i.e. if $z=x+iy$ with $x,y\in\R$, then
  $\bar{z}=x-iy$.
\item The duality relation between a vector space and its dual is denoted with angular brackets,
  e.g. $\<T,\phi\>$, $\<e^*_i,e_j\>$.
\item Inner products are denoted with parentheses, e.g. $(u,v)$.
\end{itemize}

\section{Main results}
\label{sec:main-results}

\subsection{Modelling slip and dislocations as currents}
\label{sec:disl-conf-slip}
Dislocations are usually modelled as being described by \cite{HL82,HB11}
\begin{itemize}
\item their position, a curve $\Gamma\subset\R^3$,
\item their orientation, fixed by defining a tangent field $\tau:\Gamma\to\Lambda_1\R^3$, and
\item their topological `charge', known as the Burgers vector $b\in\R^3$.
\end{itemize}
Very often, dislocations are presented as simply being described by the quantities described above,
but there is a further topological restriction on possible dislocation configurations, which
arises since dislocations must always be the boundary of a region of \emph{slip}
\cite{AO05,CGM15}. In crystal plasticity, the term slip refers to a displacement across
a surface inside a crystal such that the lattice matches perfectly on either side of it:
as a consequence, slip is characterised by a lattice vector.



To illustrate the action of slip, suppose $\mathsf{B}\in\R^{3\times3}$ is an invertible matrix,
define a fixed lattice $\Latt:=\mathsf{B}\Z^3\subset\R^3$ which describes
the structure of the material considered which satisfies the property that the shortest
non--zero lattice vector is of length 1, i.e. $\min\{|b|:b\in\Latt\setminus\{0\}\}=1$. Suppose also
that $\Sigma\subset\R^3$ is a compactly--supported oriented surface
with a normal field $\nu$ across which a slip has occurred.
The \emph{plastic distortion} corresponding to a \emph{slip vector}
$b\in\Latt$ across $\Sigma$ is then the strain field
\begin{equation}\label{eq:PlasticDef}
  z = b\otimes \nu\,\Haus^2\Rstr\Sigma,
\end{equation}
where $\Haus^2\Rstr\Sigma$  is the 2--dimensional Hausdorff measure restricted to $\Sigma$.

\begin{figure}[t]
  \centering
  \includegraphics[height=8cm]{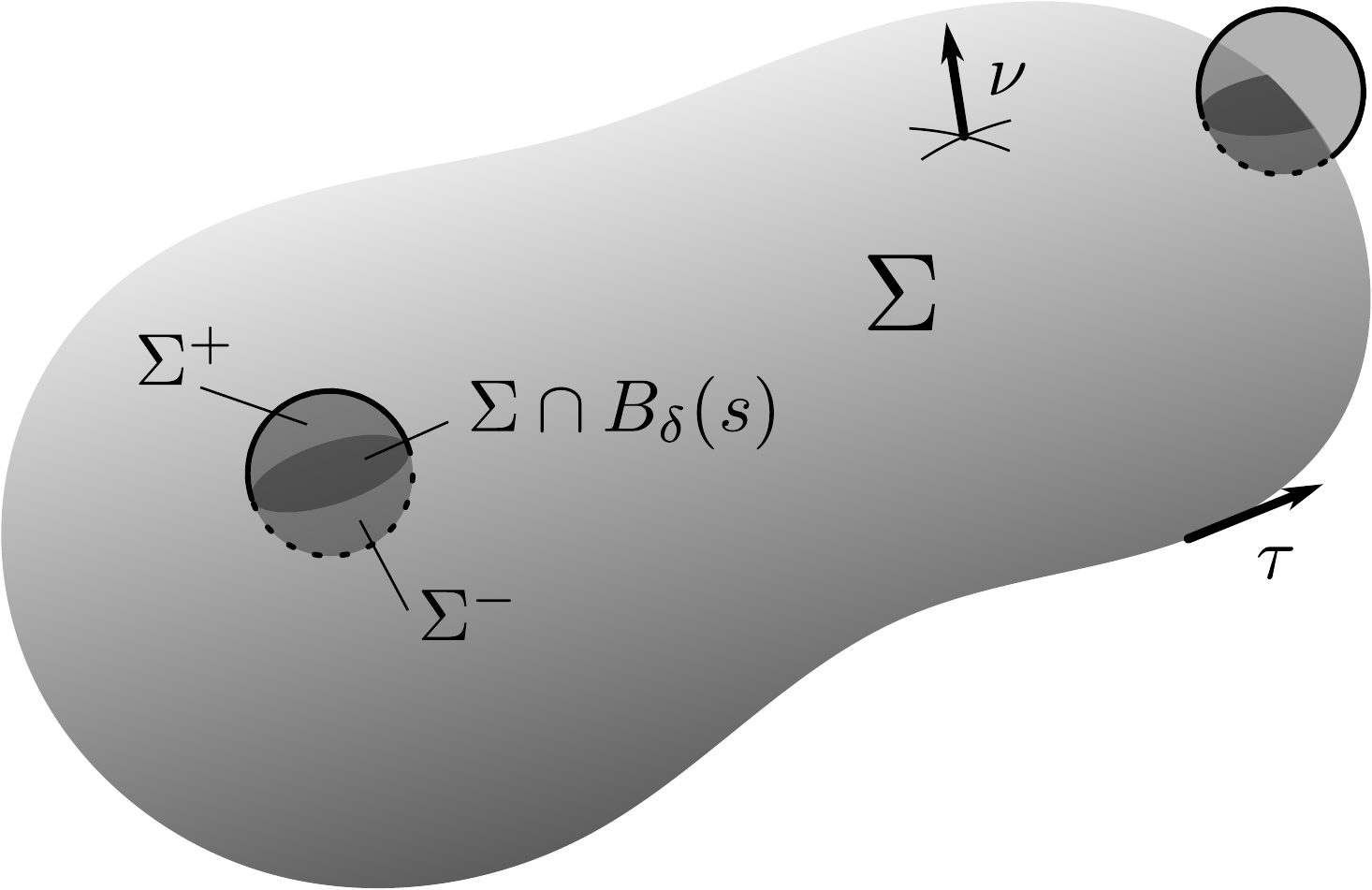}
  \caption{An illustration of the surface $\Sigma$ discussed in \S\ref{sec:smooth-distr-slip}.}
  \label{fig:Sigma}
\end{figure}

This description of $z$ as a plastic distortion is motivated by the following remark, which is
illustrated in Figure~\ref{fig:Sigma}:
suppose that $s\in\R^3\setminus\partial\Sigma$,
and $\delta>0$ is taken such that $B_\delta(s)$ is disjoint from $\partial\Sigma$. Assuming that
$\Sigma\cap B_\delta(s)$ is sufficiently regular, we may partition
$B_\delta(s)=\b(\Sigma\cap B_\delta(s)\b)\cup\Sigma^+\cup\Sigma^-$,
where $\Sigma^+$ and $\Sigma^-$ are disjoint open sets.
Define the BV vector field $u^0:B_\delta(s)\to\R^3$ such that
\begin{equation*}
  u^0(x) = 
  \begin{cases}
    b & x\in\Sigma^+,\\
    0 & x\in\Sigma^-,
  \end{cases}
  \qquad\text{for which}\qquad Du^0 = b\otimes\nu\,\mathcal{H}^2\Rstr\b(\Sigma\cap B_\delta(s)\b).
\end{equation*}
The function $u^0$ represents a jump in the displacement of $b$ across the surface
$\Sigma\cap B_\delta(s)$, and we can perform this construction for any $s\notin \partial\Sigma$;
however, the same construction fails for $s\in\partial\Sigma$, which indicates that $z$ has
non--trivial distributional curl concentrated on $\partial\Sigma$; indeed, the fact that
$z$ therefore cannot be globally represented as a gradient is precisely the reason that
$z$ is a plastic distortion.

Since the slip $z$ is concentrated on a two--dimensional set $\Sigma$, following \cite{CGM15,CGO15},
it is natural to define an associated vector--valued integral 2--current $T$, for which
\begin{equation*}
  \<T,\phi\> := \int_\Sigma \<\nu,\phi\> b\,\dd\Haus^2
  \qquad\text{for any }\phi\in\mathscr{D}^2(\R^3).
\end{equation*}
The space of such 2--currents will be denoted $\Itwo(\R^3;\Latt)$, and can be endowed with an
additive structure, which arises by taking the union of the support, and the sum of the
corresponding fields $b$.

If $T$ is an integral current, then its boundary $\partial T$ is also an integral current, and
has a consistently oriented tangent field $\tau:\partial\Sigma\to \R^3$; the equivalent of Stokes
Theorem then implies that
\begin{equation*}
  \int_{\partial\Sigma} \<\tau,\phi\>b\,\dd\Haus^1=\<\partial T,\phi\> = \< T,d\phi\>  = \int_\Sigma \<\nu,d\phi\>b\,\dd\Haus^2
  \qquad\text{for all }\phi\in\mathscr{D}^1(\R^3).
\end{equation*}
The integral 1--current $\partial T$ encodes a configuration of dislocations supported on
$\Gamma=\partial\Sigma$, with \emph{Burgers vector} $b:\Gamma\to\Latt$, and \emph{line direction}
fixed by the tangent field $\tau$. In analogue with the previous case, the space of all 1--currents
taking values in $\Latt$ is denoted $\Ione(\R^3;\Latt)$; we note that Theorem~2.5 of \cite{CGM15}
describes the structure of $\Ione(\R^3;\Latt)$.

We call the vector--valued current $T$ a \emph{slip configuration},
and $S=\partial T$ the corresponding \emph{dislocation configuration}. It is clear
that there are many possible slip configurations corresponding to the same dislocation
configuration, since there are many surfaces with the same boundary.
The space of admissible dislocation configurations is then defined to be
\begin{multline*}
  \Adm:= \B\{ S\in \Ione(\R^3;\Latt)\Bsep S = \partial T\text{ for some }T\in\Itwo(\R^3;\Latt)\text{ s.t. }\\\<T,\phi\> = \int_\Sigma \<\nu, \phi\>\,b\,\dd\Haus^2\text{ for all }\phi\in\mathscr{D}^2(\R^3)\B\}.
\end{multline*}
It is straightforward to check that $\Adm$ is an additive subspace of $\Ione(\R^3;\Latt)$, since
$\partial^2 T=0$ for any $T\in\Itwo(\R^3;\Latt)$.

The choice to require that each $S\in\Adm$ is the boundary
of a 2--current naturally encodes the fact that dislocation must be the boundary of a region of
slip, and this requirement is equivalent to that dislocation configurations satisfy a
`divergence--free' condition, as discussed in \cite{CGM15}, which is in turn equivalent to
the principle of the conservation of Burgers vector (see \cite{N52,HB11,HL82}). The additive
structure of integral currents allows for the description of complicated
configurations of dislocations via a superposition of corresponding elementary
currents.

We now define two useful functions which quantify aspects of the geometry of dislocation
configurations. The \emph{mass} of $S\in\Ione(\R^3;\Latt)$ is defined to be
\begin{equation}\label{eq:Mass}
  \Mass(S):=\sup\b\{|\<S,\phi\>| \bsep\phi\in\mathscr{D}^3(\R^3)\text{ with }|\phi(x)|\leq 1\text{ for all }x\in\R^3\b\}.
\end{equation}
In the case where $S\in\Adm$ is characterised as above, this is equivalent to the formula
\begin{equation}\label{eq:MassvsMeasure}
  \Mass(S) = \int_\Gamma |b|\,\dd\Haus^1,
\end{equation}
which corresponds physically to the total length of dislocation, weighted by the Burgers vector.
In analogy with \S9.2 in \cite{Morgan}, we also define the \emph{mass ratio} for $S\in\Adm$ as
\begin{equation}\label{eq:MassRatio}
  \MRatio(S):=\sup\bg\{\frac{\Mass\b(S\Rstr\overline{B_r(s)}\b)}{r}\,\bg|\, s\in\supp(S),r>0\bg\};
\end{equation}
here $\overline{B_r(s)}$ is the closed ball of radius $r>0$ centred at $s\in\R^3$, and
$S\Rstr A$ means the restriction of a current $S$ to a set $A$. The mass ratio should be viewed
as a way of measuring the maximal spatial density of a current; for currents of fixed mass, $\MRatio(S)$
can be arbitrarily large (see Figure~\ref{fig:MassRatio} for an explanation). Note however that
as long as $S\neq 0$, then it follows that $\MRatio(S)\geq 1$: If $S$ is supported on
$\Gamma$, then the definition of the \emph{one--dimensional density of $S$} at $s\in\Gamma$,
defined in analogy with the definition in \S9.2 of \cite{Morgan}, must satisfy
\begin{equation}\label{eq:MRatioLowerBound}
  \MRatio(S)\geq \limsup_{r\to 0}\frac{\Mass\b(S\Rstr\overline{B_r(s)}\b)}{r}\geq \min_{b\in\Latt\setminus\{0\}}|b|=1.
\end{equation}

\begin{figure}[t]
  \includegraphics[width=0.35\textwidth]{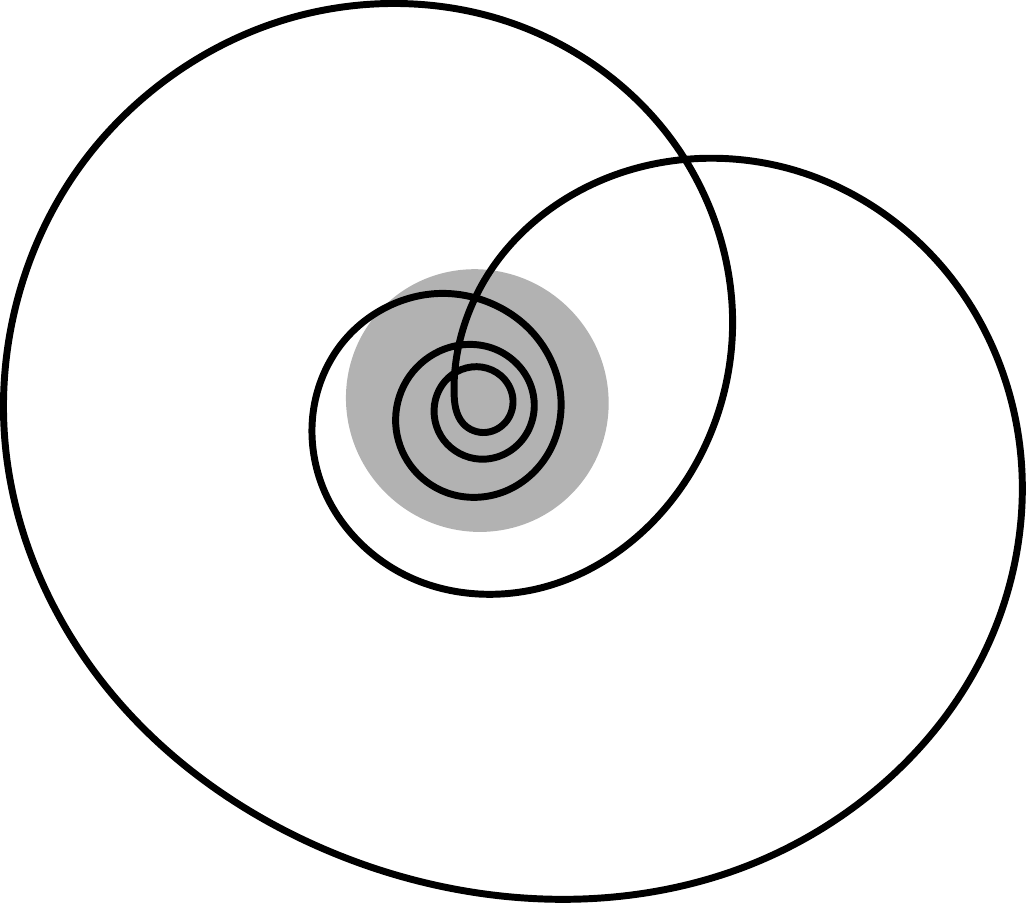}
  \caption{An example of $S\in\Adm$ with large $\MRatio(S)$. The ball in grey contains a
    large proportion of the mass. By `curling' the current tightly, the mass
    can be contained in an arbitrarily small ball, which may increase $\MRatio(S)$ while leaving
    $\Mass(S)$ fixed.}
    \label{fig:MassRatio}
\end{figure}

\subsection{Functions defined on dislocation configurations}
\label{sec:spaces-on-1-currents}
In order to formulate an appropriate setting in which to consider DDD, we will need to consider
fields which are defined on dislocations themselves. As such, we will consider functions which
are differentiable `along' a dislocation configuration $S$. This section lays out the basic
definitions we use, which allow us to state our main results.

Suppose that $S\in\Adm$, with corresponding $1$--rectifiable set $\Gamma$, Burgers vector
$b:\Gamma\to\Latt$ and tangent field $\tau$. Consider $g:\Gamma\to\R^3$ which is measurable with respect to
$\Haus^1\Rstr\Gamma$. We will say that
$\nabla_\tau g:\Gamma\to\R^3$, a $\Haus^1\Rstr\Gamma$--measurable function is a \emph{weak derivative of $g$ along $S$} if, for any $C^1$ function
$f:\R^3\to\R^3$,
we have
\begin{equation*}
  \int_\Gamma f\cdot \nabla_\tau g\,\dd\Haus^1=-\int_\Gamma Df[\tau]\cdot g\,\dd\Haus^1,
\end{equation*}
where $Df\in\BLO(\R^3;\R^3)$ denotes the Frechet derivative of $f$.
We note that if $G\in C^1(\R^3,\R^3)$ and $g:=G|_\Gamma$, $\nabla_\tau g = DG[\tau]$;
the fact that $\partial S=0$, i.e. $S$ has no boundary, ensures that no `boundary terms' are
required in the definition
above. It is straightforward to show that weak derivatives defined in this way are unique when
they exist by using the fact that $\Gamma$ is expressible as a union of images of $\R$ under Lipschitz maps, pulling back, and applying the Fundamental Lemma of the Calculus of Variations on
$\R$.

For $1\leq p<+\infty$, in analogy with the usual definitions, we set
\begin{equation*}
  \begin{aligned}
    L^p(S;\R^N)&:=\b\{g:C\to\R^N\bsep \|g\|_p<+\infty\b\}&
    &\text{with norm}&\|g\|_{L^p}&:= \bg(\int_\Gamma |g|^p\,\dd\Haus^1\bg)^{1/p},\\
    L^\infty(S,\R^N)&:=\b\{g:\Gamma\to\R^N\bsep \|g\|_{L^\infty}<+\infty\b\}&
    &\text{with norm}&\|g\|_{L^\infty}&:= \mathrm{ess}\sup \b\{|g(s)| \bsep s\in \Gamma\},
  \end{aligned}
\end{equation*}
where the essential supremum in the latter definition is taken up to $\Haus^1$ null sets.

We also define the space
\begin{equation*}
  H^1(S;\R^3):=\b\{g\in L^2(S;\R^3)\bsep \nabla_\tau g\in L^2(S;\R^3)\b\},
\end{equation*}
which has a real Hilbert space structure when endowed with the inner product
\begin{equation*}
  (f,g) = \int_\Gamma \B[f\cdot g+\nabla_\tau f\cdot \nabla_\tau g\B]\,\dd\Haus^1.
\end{equation*}
As usual, we will denote the dual space of $H^1(S;\R^N)$ as $H^1(S;\R^N)^*$.

\subsection{Energy of regularised slip distributions}
\label{sec:smooth-distr-slip}
For any slip configuration $T\in\Itwo(\R^3;\Latt)$, we now define a regularised energy by a procedure
similar to that considered in \S2.4.2 of \cite{CGO15}, distributing slip about the set $\Sigma$
on which $T$ is supported by mollifying.
Physically--speaking, our choice to smooth distributions of slip can be justified by noting that
a definition
of the lattice plane over which slip has occurred cannot be given with a precision greater than
that of a single lattice spacing, and similarly, the position of a dislocation core cannot 
be ascertained to a precision of less than a few lattice spacings. As such, regularising
defines the lengthscale at which linear elasticity is invalid, and as stated in
\S\ref{sec:regul-theory-ddd}, this choice has the convenient mathematical
benefit that all fields considered are non--singular, in common with the reality on the
atomistic scale.

To this end,
we suppose that $\varphi^1\in C^\infty(\R^3)$ is a function which is rapidly decreasing (in the Schwartz sense), is radially symmetric, and satisfies $\int_{\R^3}\varphi^1(x)\dx = 1$, such as
the Gaussian
\begin{equation}\label{eq:Gaussian}
  \varphi^1(x)=\frac{1}{(2\pi)^{3/2}}\exp\left(-\smfrac{1}2|x|^2\right).  
\end{equation}
For any $\eps>0$, we then define $\varphi^\eps(x):=\eps^{-3}\varphi^1(x/\eps)$.

If $T\in\Itwo(\R^3;\Latt)$ is a slip configuration supported on $\Sigma$, with corresponding normal
field $\nu$, and slip $b:\Sigma\to\Latt$,
in analogy with the plastic deformation considered in \eqref{eq:PlasticDef},
we define a \emph{smoothed plastic distortion} $z^\eps_T:\R^3\to\R^{3\times3}$ to be
\begin{equation}\label{eq:zepsDef}
  z^\eps_T(x):= \int_\Sigma b(s)\otimes\nu(s)\varphi^\eps(x-s)\dd\Haus^2(s).
\end{equation}
Here $\Haus^2$ is the 2--dimensional Hausdorff measure (i.e. the surface area measure on $\Sigma$),
and the field $z^\eps_T$ is well--defined and smooth since $\varphi^\eps$ is assumed to be $C^\infty$.

Following the ideas of Kr\"oner and Mura \cite{Kroner58,Kröner1959,M91}, we suppose that the system
equilibriates elastically in response to this plastic strain.
Using the additive decomposition of the strain, the total distortion
$\beta^\eps$ is assumed to take the form
\begin{equation*}
  \beta^\eps = z^\eps_T+Du^\eps,
\end{equation*}
where $u^\eps:\R^3\to\R^3$, and recalling the definition of $\dot{H}^1(\R^3)$ given in
\S\ref{sec:notation-conventions}, the energy at equilibrium due to a configuration of slip
described by $T\in\Itwo(\R^3;\Latt)$ is therefore
\begin{equation}\label{eq:SlipEnergy}
  \E^\eps(T):=\min_{u^\eps\in\dot{H}^1(\R^3)} \I(z^\eps_T+Du^\eps),
\end{equation}
where the total internal energy, $\I:L^2(\R^{3\times3})\to\R$, is defined to be
\begin{equation}\label{eq:Energy}
  \I(\beta):=\int_{\R^3} \smfrac12\beta : \Elas\,\beta\dx = \int_{\R^3} \smfrac12\Elas_{ijkl}\beta_{ij}\beta_{kl}\dx.
\end{equation}

The fundamental insight in the work of Kr\"oner and Mura is that while dislocations must be the
boundary of a region of slip, the precise surface over which slip has occurred is irrelevant to
the energy and mechanical response of the system, i.e. although \eqref{eq:SlipEnergy} appears to
depend on $T$, in fact it only depends upon $\partial T$. This is compatible with the
experimental observation that dislocations moving through a crystal leave no trace of having
passed, since the crystal `heals' perfectly after slip occurs. Using these insights, we prove the
following theorem which encodes the dependence of the energy on $\partial T$ alone, and
moreover provides an expression for the internal energy given directly in terms of an integral
over the dislocation configuration itself.

\begin{theorem}\label{th:Energy}
  If $T\in\mathscr{I}_2(\R^3;\Latt)$ is compactly supported, the energy $\E^\eps$ defined in
  \eqref{eq:SlipEnergy} depends only on $S=\partial T$, so that we may define
  $\Phi^\eps:\Adm\to\R$ with
  \begin{equation*}
    \Phi^\eps(S):=\E^\eps(T)\quad\text{for any }S\in\Adm\text{ where }S=\partial T.
  \end{equation*}
  Further, if $T\in\Itwo^3$ and $S = \partial T\in\Adm$ respectively take the form
  $\<T,\phi\> = \int_{\Sigma}\<\nu,\phi\>b\,\dd\Haus^2$ for $\phi\in \mathscr{D}^2(\R^3)$ and
  $\<S,\eta\> =  \int_{\Gamma}\<\tau,\eta\>b\,\dd\Haus^1$ for $\eta\in\mathscr{D}^1(\R^3)$,
  then the energy functionals $\E^\eps(T)$ and $\Phi^\eps(S)$ may be written as
  \begin{subequations}
\begin{align}\label{eq:JRep}
\begin{split}
  \E^\eps(T) &= \int_{\Sigma\times\Sigma}\smfrac12\Jker_{abcd}(s-t)b_a(s)\nu_b(s)b_c(t)\nu_d(t)\dd(\Haus^2\otimes\Haus^2)(s,t)
\end{split}\\
  \label{eq:KRep}
  \begin{split}
    \Phi^\eps(S)&= \int_{\Gamma\times\Gamma}\,\smfrac12\Kker_{abcd}(s-t)b_a(s)\tau_b(s)b_c(t)\tau_d(t)
    \dd(\Haus^1\otimes\Haus^1)(s,t),
  \end{split}
\end{align}
\end{subequations}
where the kernels $\Jker,\Kker:\R^3\to\R^{3\times3\times3\times3}$ are defined to be
\begin{subequations}\label{eq:KerDefs}
\begin{align}
  \Jker_{kmgr}(s)&:=\int_{\R^3}\Elas_{abcd}\Alt_{bpl}\Elas_{ijkl}\GMat^\eps_{a i,jn}(x-s)\Alt_{pmn}
                       \Alt_{dqh}\Elas_{efgh}\GMat^\eps_{c e,fs}(x)\Alt_{qrs}\dx,\label{eq:SSKer}\\
  \Kker_{kpgq}(s)&:=\int_{\R^3}\Elas_{abcd}\Alt_{bpl}\Elas_{ijkl}\GMat^\eps_{a i,j}(x-s)\Alt_{dqh}\Elas_{efgh}
                        \GMat^\eps_{c e,f}(x)\dx,\label{eq:LLKer}
\end{align}
\end{subequations}
where $\Alt$ is the alternating tensor as defined in \S\ref{sec:notation-conventions}, and
$\GMat^\eps:=\GMat *\varphi^\eps$, i.e. the convolution of the elastic Green's function with
$\varphi^\eps$.
Moreover, $\Jker$ and $\Kker$ satisfy the following properties.
  \begin{enumerate}
  \item $\Jker_{abcd}(s)=\Jker_{cdab}(s)=\Jker_{abcd}(-s)$ and
    $\Kker_{abcd}(s)=\Kker_{cdab}(s)=\Kker_{abcd}(-s)$ for any $s\in\R^3$;
  \item $\Jker$ and $\Kker$ are smooth.
  \item For any $m\in\N$, $0\leq j\leq m$, and vectors $v_1,\ldots,v_j\in\Sph^2$,
    there exists a constant $C_{m,j}$ such that for all $s\in\R^3$,
    \begin{equation*}
      \B|D^m \Kker(s):\B[v_1,\ldots, v_j,\frac{s}{|s|},\ldots,\frac{s}{|s|}\B]\B|
      \leq \frac{C_{m,j}}{\sqrt{\eps^{2m+2}+\eps^{2j}|s|^{2m+2-2j}}}.
    \end{equation*}
  \end{enumerate}
  
\end{theorem}\medskip

\noindent
The proof of this theorem is given in \S\ref{sec:elast-energy-disl}, and proceeds by verifying
that $\E^\eps(T)$ as given in \eqref{eq:SlipEnergy}
is well--defined, before characterising the solution of the minimisation problem by applying
the elastic Green's function and the ideas used to derive \emph{Mura's formula} \cite{M91}.
Ideas from \cite{BBS80} (closely related to ideas used in proving Theorem~4.1 in \cite{CGO15})
are then used to derive a Fourier characterisation, which allow us to deduce the asserted
properties of the kernels $\Jker$ and $\Kker$.

We note at this stage that the definition of the regularised energy $\E^\eps$ clearly depends upon
the precise choice of $\varphi^\eps$, which particularly influences the behaviour of segments of
dislocation line which are at a distance on the order of $\eps$ apart. As argued in the
introduction,
this is exactly the scale at which linear elasticity theory fails to be valid, but the
results of \cite{CGO15} demonstrate that, to leading order as $\eps\to0$, the $\Gamma$--limit of
the energy is independent of the choice of regularisation. While this theory (and indeed no
continuum theory) can therefore accurately capture the behaviour of dislocations on the lattice
scale, we can hope to accurately capture the behaviour of dislocation loops which are `large'
relative to the interatomic distance.
Furthermore, to capture additional knowledge about the crystal structure considered, the choice to
make $\varphi^1$ radially symmetric and independent of $b$ and $\nu$ could be relaxed, allowing
for the modelling of different `widths' of plane over which slip occurs, although we do not pursue
this choice here.

\subsection{The Peach--Koehler force}
The power of the expression for the energy given in \eqref{eq:KRep} is that it allows us to take
explicit variations of the energy, and thereby to derive the equivalent of the
configurational Peach--Koehler force in this model.
Since we are varying the
set on which the Burgers vector $b$ and line direction $\tau$ are defined, an appropriate notion
of variation is that of \emph{inner variation} (see Chapter~3 of \cite{GiaquintaHildebrandtPartI})
or \emph{variation of the reference state} (as described in \S2.1.5 of \cite{GMSVol2}).
To construct an inner variation, we use the notion of pushforward
(defined in \S\ref{sec:pushforward-pullback}): if $g\in C^{0,1}(\R^3;\R^3)$, we define the
\emph{inner variation}  of $\Phi^\eps$ at $S\in\Adm$
in the direction $g$ to be the linear functional
\begin{equation}\label{eq:InnerVariation}
  \<D\Phi^\eps(S),g\>:= \frac{\dd}{\dd\del}\Phi^\eps\b((\id+\delta g)_\#S\b)\bg|_{\del=0}
\end{equation}
where $\id:\R^3\to\R^3$ is the identity mapping $\id(x):=x$.
The result of taking this variation and properties of the resulting functional are encoded in the
following theorem, which provides an expression of the variation as a field defined on the
dislocation configuration, as well as a uniform bound and a form of continuity result under
deformation.

\begin{theorem}
  \label{th:Force}
  If $S\in\Adm$, the inner variation of the energy at $S$ is given by
  \begin{equation}
    \label{eq:PKForce}
    \begin{gathered}
      \<D\Phi^\eps(S),g\> = -\int_\Gamma \PK_i(s,S)g_i(s)\dd\Haus^1(s),\quad\text{where}
      \quad \PK(s,S):=G(s,S)\wedge \tau(s)\\
      \text{and}\quad G_k(s,S):=\int_\Gamma\Alt_{klm}\Kker_{alcd,m}(s-t)
      b_a(s)b_c(t)\tau_d(t)\dd\Haus^1(t).
    \end{gathered}
  \end{equation}
  If $S = \partial T$, where $T\in\Itwo(\R^3;\Latt)$ is supported on $\Sigma$ with Burgers vector
  $b:\Sigma\to\Latt$ and normal field
  $\nu$, $G$ can alternatively be written
  \begin{equation}\label{eq:PKForceSurface}
    G_k(s,S) = -\int_\Sigma\Alt_{def}\Alt_{klm}\Kker_{alcd,mb}(s-t)b_a(s)b_c(t)\nu_f(t)\dd\Haus^2(t).
  \end{equation}
  Moreover, recalling the definitions of $\Mass$ and $\MRatio$ given respectively in
  \eqref{eq:Mass} and \eqref{eq:MassRatio}, we have the bound
  \begin{equation}\label{eq:PKLInfty}
      \b\|\PK(S)\b\|_{L^\infty}\leq
    \frac C\eps\|b\|_{L^\infty}\MRatio(S)\log\bg|1+\frac{2\,\Mass(S)}{\eps\,\MRatio(S)}\bg|,
  \end{equation}
  and if $g:S\to\R^3$ is a Lipschitz map and $F:=\id+g$, then
  \begin{equation}\label{eq:PKContinuity}
    \b\|F^\#\PK(F_\#S)-\PK(S)\b\|_{L^\infty}\leq \b(1+C\Mass(S)\b)\|\nabla_\tau g\|_{L^\infty}+C\Mass(S)\|g\|_{L^\infty},
  \end{equation}
  where $C$ is a constant independent of $v$ and $S$.
\end{theorem}
\medskip

\noindent
A proof of this Theorem~\ref{th:Force} is given in \S\ref{sec:deform-disl-peach}.
The main achievement of this result is in obtaining the bound \eqref{eq:PKLInfty}, which is crucial
to proving the long--time existence result for the formulation of DDD considered in the following
sections. This bound is a significant improvement over the more na\"ive estimate
$\|\PK(S)\|_{L^\infty}\leq C\,\Mass(S)$, which can be deduced
directly from the fact that $\Kker$ and $D\Kker$ are uniformly bounded, proved in
Theorem~\ref{th:Energy}. Instead, the estimate is obtained by a rearrangement argument, using
the limited geometric information $\MRatio(S)$ provides to guarantee this weaker dependence on the
mass.

The bound \eqref{eq:PKContinuity} provides a form of continuity for the Peach--Koehler force,
demonstrating that under a Lipschitz variation of $S$ which is close to the identity,
the Peach--Koehler force of the new configuration $F_\#S$ is close to that obtained on the initial
configuration. Since these fields are defined on different sets, in order to measure the difference
we must pull back this new field onto the initial configuration.

\subsection{Dislocation mobility}\label{sec:dislocation-mobility}
It is widely believed that moving dislocations dissipate energy via phonon radiation (see for example \S3.5 in
\cite{HB11}, and \S7-7 in \cite{HL82}): the movement of a dislocation through a `rough' landscape of local
minima generates high--frequency lattice waves which radiate away as heat, leading to drag.
When modelling dislocation motion at low to moderate strain rates, it is typically assumed that
drag dominates inertial effects, which are therefore neglected. This assumption has been supported
by microscopic simulations of dislocation motion (see for example the discussion in \S4.5 and
\S10.2 of \cite{BC06}, and \cite{CCBY01,CCBY02,BGE06}). 
Neglecting inertia inevitably entails that dislocation motion
(and hence plastic distortion) is modelled as a dissipation--dominated process;
as such, a natural mathematical framework for modelling dislocation motion is that of a \emph{generalised gradient flow} \cite{AGS}.

At low temperatures, the process of slip is dominated by \emph{glide}, a process which
allows dislocations to move while conserving lattice volume \cite{HL82,AO05,HB11}. Requiring that dislocations undergo
glide motion only is equivalent to requiring that slip can only evolve on planes which contain $b$.
The evolution of slip in directions parallel to the Burgers vector is called \emph{climb}, and
requires mass transport via point defect diffusion; at low temperatures this is a
much slower process than that of glide.

To model these phenomena in DDD simulations, various constitutive assumptions on dislocation
mobility are available; for various examples, see \cite{CB04,ACetal07,HB11}. Usually,
the \emph{velocity}
$v$ of a segment of dislocation is related to the configurational force on the dislocation line
via a mobility function $\mathcal{M}$, which depends locally on the Burgers vector $b$, the
dislocation orientation $\tau$, and the Peach--Koehler force $\PK$:
\begin{equation*}
  v = \mathcal{M}(b,\tau,\PK).
\end{equation*}
Such mobilities are informed by molecular dynamics simulations or experiment, and are generally
linear, power laws, or possibly include some frictional threshold before the onset of dislocation
motion, mimicking the Peierls barrier. In common with many geometric evolution problems, it is
usually assumed that dislocations have velocity only in normal directions, i.e.
$\b(\mathcal{M}(b,\tau,f),\tau\b)=0$ for any $b,\tau$ and $f$.

In order to define a generalised gradient flow framework which encompasses mobilities of the type
referred to above, we restrict ourselves to considering mobilities which take the form
\begin{equation*}
  \mathcal{M}(b,\tau,f) := -\nabla_f\psi^*(b,\tau,-f),
\end{equation*}
where $\psi^*$ is a \emph{entropy production} which is convex in $f$, describing the rate at
which entropy is produced by the force $f$. Requiring the existence of $\psi^*$ is not particularly
restrictive, as it ensures that energy is conserved in a closed system, and to the author's
knowledge, all mobility laws for DDD used in practice take this form.

When an entropy production is defined, it is natural to define a conjugate
\emph{dissipation potential} $\psi$,
which describes the rate at which energy is lost through the variation of a dislocation
configuration, and is given as the Legendre--Fenchel transform of $\psi^*$, i.e.
\begin{equation}\label{eq:EntropyProduction}
  \psi(b,\tau,v) := \sup_{f\in\R^3}\b\{\<v,f\>-\psi^*(b,\tau,f)\b\}.
\end{equation}
The frictional force resulting from a velocity $v$ is then $-\nabla\psi(v)$; this
follows from the fact that for convex conjugate functions,
\begin{equation}
  v=-\nabla\psi^*(-f)\quad\text{if and only if}\quad f=-\nabla\psi(-v)\quad\text{if and only if}\quad
  \<f,v\>=\psi(-v)+\psi^*(-f).\label{eq:ConConjRel}
\end{equation}
As an example, in the case of a linear relationship between configurational force and velocity,
$v=B(b,\tau)f$, it is straightforward to check that we may define
\begin{equation*}
  \psi^*(f) = \tfrac12\<f, B(b,\tau) f\>
  \qquad\text{and}\qquad\psi(v) = \tfrac12\<v,B(b,\tau)^{-1}v\>,
\end{equation*}
where $B^{-1}$ denotes the matrix inverse of $B$.

\begin{figure}[t]
  \includegraphics[width=0.5\textwidth]{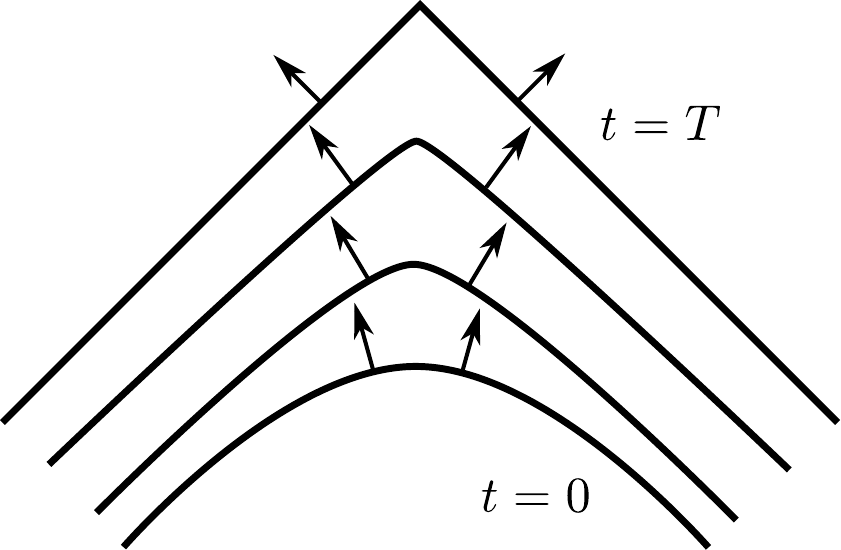}
  \caption{Development of a corner in finite time if dissipation potential depends on $v$ alone. Arrows reflect instantaneous velocity of the dislocation.}
  \label{fig:FiniteTimeCorner}
\end{figure}

In practice, we find that requiring that the dissipation potential is uniquely a function of
the dislocation velocity as in \eqref{eq:EntropyProduction}
is insufficient to guarantee a well--defined evolution. As an illustration,
consider the possible scenario in Figure~\ref{fig:FiniteTimeCorner}.
While the dislocation is initially smooth,
and $\PK$ is smooth up until the final time, $v=B(b,\tau)\PK$ loses regularity exactly as a
jump in $\tau$ develops. At the final time, the dynamics could require $v$ to be discontinuous,
resulting in segments of dislocation `ripping apart', breaking the physical requirement
that dislocations are the boundary of regions of slip, and leading to a blow--up of the evolution.

To avoid this possibility, in the following section, we introduce assumptions requiring that
the energy dissipated by dislocation motion also depends on the rate of change of the tangent to
the dislocation.


\subsection{Dissipation potential}
\label{sec:Constitutive}
Motivated by the discussion in \S\ref{sec:dislocation-mobility}, we suppose that the velocity of a
dislocation is described by a field along its length, $v$, and assume that the dissipation
potential for a dislocation configuration $S$ supported on $\Gamma\subset\R^3$ with a velocity
field $v:\Gamma\to \R^3$
is expressed as
\begin{equation*}
  \Psi(S,v)
  = \int_\Gamma \tfrac12 \nabla_\tau v\cdot A(b,\tau)\nabla_\tau v +\psi(b,\tau,v)\,\dd\Haus^1
\end{equation*}
where, $\nabla_\tau v$ denotes the weak derivative of the velocity along the dislocation line
as defined in \S\ref{sec:spaces-on-1-currents},
and as usual, $\tau:\Gamma\to\Sph^2$ is the tangent field to $S$.
We make the following constitutive assumptions on $A$ and $\psi$.
\begin{itemize}
\item[\bf (C$_1$)] $A:\Latt\times\Sph^2\to\R^{3\times 3}$ is a matrix--valued
  function which is symmetric and strictly positive definite everywhere, and
  there exists $\alpha>0$ such that
  \begin{equation*}
    w\cdot A(b,\tau)w\geq\alpha|w|^2\qquad\text{for any }(b,\tau,w)\in\b(\Latt\setminus\{0\}\b)\times\Sph^2\times\R^3.
  \end{equation*}
\item[\bf (C$_2$)] $\psi:\Latt\times\Sph^2\times\R^3\to[0,+\infty]$ is a positive function
  which is strictly convex in its third argument, satisfying $\psi(b,\tau,0)=0$
  for any $(b,\tau)\in\Latt\times\Sph^2$. Moreover,
  \begin{equation*}
    \psi(b,\tau,v)=+\infty\qquad
    \text{if}\quad v\cdot\tau\neq0.
  \end{equation*}
\item [\bf (R)] For any $b\in\Latt$, $\tau\mapsto A(b,\tau)$ is smooth and
  bounded on $\Sph^2$, and $(\tau,v)\mapsto \psi(b,\tau,v)$ is smooth on
  the set $\{v\cdot\tau\in\Sph^2\times\R^3:v\cdot\tau\neq0\}$.
\item[\bf (G)] There exists $\beta>0$ such that
  \begin{equation*}
    \quad \psi(b,\tau,v)\geq \tfrac12\beta|v|^2\qquad
    \text{for all }(b,\tau,f)\in\Latt\times \Sph^2\times\R^3.
  \end{equation*}
\end{itemize}
As remarked in the previous section, the assumptions above are broad enough to encode a wide variety of modelling assumptions made when
modelling dislocation dynamics:
\begin{itemize}
\item The convexity assumptions {\bf(C$_1$)} and {\bf(C$_2$)} imply that dissipation
  potential is always positive and increases as the dislocation velocity or rate of bending
  increases, and no change in the energy dissipated is produced if dislocations do not deform or translate
  within the material. The fact that $\psi(b,\tau,v)=+\infty$ if $v\cdot\tau\neq 0$
  enforces the requirement that meaningful dislocation velocity fields must be locally
  perpendicular to the dislocation line.
\item The regularity assumption {\bf(R)} ensures that energy dissipation rate varies
  smoothly with the velocity and orientation of the dislocation line.
\item The growth assumption {\bf(G)} is a technical assumption, and is not particularly
  restrictive, however relaxing it would make some aspects of our analysis more technical.
\end{itemize}
On the other hand, the choice to make the dissipation depend upon $\nabla_\tau v$ as well as $v$
appears to be a novel addition to DDD. It does not seem to be unreasonable
to require that energy is dissipated by the bending or stretching of dislocation lines, and it
this additional term gives us control of the bending rate, ruling out the
generation of singularities similar to those shown in the figure
and allowing us to prove that the evolution is well--posed.
It would be of great interest to understand whether this term is indeed
physically--justified via a future computational study of a realistic model, or indeed whether this
additional dependence can be mathematically removed via (for example) a vanishing viscosity
argument.

As an indicative example of a constitutive relation satisfying these assumptions which is closely
related to a dislocation mobility which is already present in the literature, we may define
\begin{equation*}
  \Psi(S,v) = \int_S \tfrac12 \alpha |\nabla_\tau v|^2
  +\psi(b,\tau,v)\,\dd\Haus^1(s),
\end{equation*}
so that $A(b,\tau)= \alpha\,\Id$, and where we set
\begin{gather*}
\psi(b,\tau,v) = 
  \begin{cases}
    \tfrac12v^TB^\dagger(b,\tau)v & v\cdot\tau=0\\
    +\infty & v\cdot\tau\neq0,
  \end{cases}
  \quad\text{with}\\
  B(b,\tau) = \bigg(\frac{|b\wedge\tau|^2}{B_{eg}^2}+\frac{(b\cdot\tau)^2}{B_s^2}\bigg)^{-\frac12}\frac{\Proj(\tau)b\otimes \Proj(\tau)b}{|b\wedge\tau|^2}+
  \sqrt{B^2_{ec}|b\wedge\tau|^2+B_s^2(b\cdot\tau)^2}\frac{(b\wedge\tau)\otimes (b\wedge\tau)}{|b|^2|b\wedge\tau|^2}.
\end{gather*}
Here, $B^\dagger$ denotes the Moore--Penrose pseudo--inverse of $B$; in this case, this is simply the
matrix which has the same eigenspaces as $B$ and inverts any non--zero eigenvalues.
$B_{eg}$, $B_{ec}$, $B_s>0$ are all mobility parameters describing dissipative timescales resulting from various different modes of motion
(bending and stretching, glide of edge dislocations, climb of edge dislocations and
glide of screw dislocations), and $\alpha^{-1}>0$ is the energy dissipation rate per 
unit additional area swept out per unit length of dislocation.

Defining $\psi^*(b,\tau,f) = \sup\{v\cdot f-\psi(b,\tau,v)\sep v\in\R^3\}$, i.e. the
Legendre--Fenchel transform of $\psi$, we find that
\begin{equation*}
  \psi^*(b,\tau,f) = \tfrac12 \b(f,B(b,\tau)f\b),
\end{equation*}
and we note that $v = B(b,\tau)f = -\nabla\psi^*(b,\tau,-f)$ is the mobility relation for
dislocations in BCC defined in equation~(10.40) of \cite{BC06}. We therefore see that if
$\alpha=0$, the model as defined above reduces to that considered in \cite{BC06}.

We remark that as a consequence of assumption {\bf (C$_2$)},
we have the following characterisation of the subgradient of $\psi$:
\begin{equation}\label{eq:Subgradient_psi^*}
  \partial_v\psi(b,\tau,v) = \begin{cases}
    \emptyset & v\cdot\tau\neq 0,\\
    \{D_\tau^\perp\psi(b,\tau,v)\} &v\cdot\tau=0,
  \end{cases}
\end{equation}
where $D_\tau^\perp\psi(b,\tau,v)$ means the gradient of $\psi$ taken in directions perpendicular
to $\tau$.

\subsection{Evolution problem}
When viewed as a function defined on $H^1(S;\R^3)$, $\Psi(S,\cdot)$ has a well--defined
subdifferential with respect to $v$, $\partial_v\Psi(S,v)\subset H^1(S,\R^3)^*$,
and recalling the definition of $D_\tau^\perp\psi$ made above,
$\xi\in\partial_v\Psi(S,v)$ implies that
\begin{equation*}
  \<\xi,w\> = \int_S \nabla_\tau v\cdot A(b,\tau)\nabla_\tau w
  +D_\tau^\perp\psi(b,\tau,v) \cdot w\,\dd\Haus^1
\end{equation*}
as long as $v\cdot\tau = 0$ $\Haus^1$--almost everywhere on $S$ (otherwise,
$\partial_v\Psi(S,v)=\emptyset$). As shorthand for the formula above, we will write
\begin{equation*}
  \partial_v\Psi(S,v) = -\div_\tau\b[A(b,\tau)\nabla_\tau v\b]+ D_\tau^\perp\psi(b,\tau,v).
\end{equation*}
This expression corresponds to the frictional force induced on the dislocation configuration by
moving according to the velocity field $v$. In the overdamped regime where inertial effects
are neglected, these frictional forces balance with the configurational forces,
so that $v$ must satisfy
\begin{equation}\label{eq:StrongForm}
  -\div_\tau\b[A(b,\tau)\nabla_\tau v\b]+ D_\tau^\perp\psi(b,\tau,v)= f^{\mathrm{PK}}\quad\text{in }H^1(S;\R^3)^*.
\end{equation}
It is straightforward to check that, since $\Psi$ is convex on $H^1(S;\R^3)$, 
an equivalent criterion is to require that $v$ solves the minimisation problem
\begin{equation}\label{eq:TstepMinProb}
  v\in\argmin_{v\in H^1(S;\R^3)} \Psi(S,v)+\<D\Phi^\eps(S),v\> = \argmin_{v\in H^1(S;\R^3)} \Psi(S,v)-(\PK,v)_{L^2}.
\end{equation}

The force balance equation \eqref{eq:StrongForm} is \emph{Eulerian} in nature: it must be satisfied
on the dislocations themselves as they deform. For the purpose of proving existence of
solutions, we now cast an alternative \emph{Lagrangian} formulation for a solution of DDD:
we suppose that the position of points on the dislocation line at time $t$ is expressed as a
function of time and position on the dislocation line at the initial time.
In the language of geometry, this idea is expressed as a
\emph{pushforward by $U$} (recall \S\ref{sec:pushforward-pullback}), where
$U:[0,T]\times S^0\to\R^3$.
Writing $U(t)$ to denote the mapping at time $t$, we require that
\begin{equation*}
  U(0) = \id\qquad\text{on }S^0,
\end{equation*}
and the `trajectory' of currents is then
\begin{equation*}
  S^t:=U(0)_\#S^0\qquad\text{for all }t\in[0,T].
\end{equation*}
In this formulation, we see that $U$ directly identifies the family of currents
$\{S^t\}\subset\Adm$. Thanks to \S4.1.14 of \cite{Federer}, $S^t$ is well--defined as a current
as long as $U^t$ is a Lipschitz map on $S^0$.

We will find that an appropriate `energy space' in which to seek to prove the existence of $U$ is
$H^1([0,T];H^1(S^0;\R^3))$, i.e. the space of $L^2$ Bochner--integrable functions from $[0,T]$ in
$H^1(S^0;\R^3)$ with $L^2$ Bochner--integrable weak derivatives; for further
detail on the definition of such spaces, see for example Chapter~7 of
\cite{Roubicek}, and we denote the weak time derivative of $U(t)$ as  $\dot{U}(t)$.

\emph{A priori}, there is no guarantee that a generic $U(t)\in H^1(S^0;\R^3)$ is
Lipschitz, and therefore no guarantee that $U(t)_\# S^0$ is a current.
In order to ensure $S^t\in\Adm$ for all time, we additionally require that
$U(t)\in\CC^{0,1}(S^0;\R^3)$.
We will say that a pair $(U,\{v^t\}_{t\in[0,T]})$ form a \emph{solution of DDD} as formulated
in \eqref{eq:StrongForm} if
$U\in H^1([0,T];H^1(S^0;\R^3))$, $U(t)\in C^{0,1}(S^0;\R^3)$ and $v^t\in H^1(U(t)_\#S^0;\R^3)$ for almost every $t\in[0,T]$, and
\begin{equation}\label{eq:DDD}
  \begin{aligned}
  \partial_v\Psi\b(U(t)_\#S^0,v^t\b) &\ni \PK\b(U(t)_\#S^0\b) &&\text{in }H^1\b(U(t)_\# S^0;\R^3\b)^*\\
  \quad\text{and}\quad U(t)^\#v^t&=\dot{U}(t)&&\text{in }H^1(S^0;\R^3)
\end{aligned}
\end{equation}
for almost every $t\in[0,T]$, where the definitions of pushforward $U(t)_\#$ and pullback $U(t)^\#$
are given in \S\ref{sec:pushforward-pullback}.
Now that we have given a definition of what it means to be a solution to DDD,
we have the following well--posedness result.

\begin{theorem}\label{th:WellPosedness}
  If $S^0\in\Adm$ is a finite union of $C^{1,\gamma}$ curves with $\gamma\in(0,1]$, and satisfies
  $\MRatio(S^0)<+\infty$, then there exists a unique solution $(U,\{v^t\})$ satisfying
  \eqref{eq:DDD} for $t\in[0,T]$ where
  \begin{equation*}
    T:=\sup\b\{t\in\R:\MRatio(U(s)_\#S^0)<+\infty\text{ for all }s\leq t\b\}.
  \end{equation*}
  Moreover, $U\in C^0\b([0,T];C^1(S^0;\R^3)\b)$.
\end{theorem}\medskip

\noindent
We note that a strength of this result is that it allows for dislocation collision, since there is
no requirement that $U$ is invertible, although as currently phrased in a Lagrangian form, it
is not clear that dislocations satisfy the correct evolution if they merge. Indeed, in practical
simulations of DDD, dislocations are remeshed exactly as dislocation
segments approach separation distances of $O(\eps)$, as discussed in \S10.4 of
\cite{BC06}. It would be of great interest to understand how best to correctly incorporate this
phenomenon into a mathematical theory of DDD in future.

Generically, we expect $\MRatio(U(t)_\#S^0)<+\infty$ for all time, since the contrary would require
a concentration of internal energy on a small set; while at present we are unable to rule out the
possibility of this occurring, it would be interesting in future to confirm existence for all time
for at least a large class of initial data.

\subsection{Conclusion}
We have provided a framework in which to study a regularised form of DDD in three
dimensions, inspired by various ideas in both the Engineering and Mathematics literature
\cite{BBS80,M91,CAWB06,CGM15,CGO15,SvG16a}. Computable integral formulae for the energy
of and configurational forces on a general dislocation configuration were derived,
and bounds on the configurational force which depend weakly on the overall length of dislocation
were obtained. A gradient flow formalism in which to study DDD was proposed,
and within this framework, we obtained a well--posedness result for the
evolution up until the first time an infinite density of dislocations develops.

It is hoped that the energetic framework developed here is sufficiently general to open 
the way to new upscaling results such as those in \cite{P07,GLP10,DLGP12,SZ12,MSZ14,MSZ15}
in a three--dimensional setting, to build upon the results of \cite{CGO15} in the case where
the regularisation lengthscale $\eps$ tends to zero, and to allow the mathematical study of the
numerical schemes used in practical implementations of DDD.

\section{The energy of dislocations}
\label{sec:elast-energy-disl}

Our aim in this section is to prove Theorem~\ref{th:Energy}, providing a characterisation of the
energy of dislocations.

\subsection{Minimisation problem}
\label{sec:elast-resp-slip}
Our first step towards proving Theorem~\ref{th:Energy} is to characterise the
solution to the minimisation problem \eqref{eq:SlipEnergy}.

\begin{lemma}
  \label{th:DistortionCharacterisation}
  Assuming that $\Elas$ satisfies a Legendre--Hadamard condition, $T\in\Itwo(\R^3;\Latt)$ and
  $z^\eps_T$ is as defined in \eqref{eq:zepsDef},
  the variational problem in \eqref{eq:SlipEnergy} has a unique solution, which is smooth and satisfies the equation
  \begin{equation}\label{eq:ElasticResp}
    -\Elas_{ijkl}u^\eps_{k,lj}=\Elas_{ipqr}(z^\eps_T)_{qr,p},
  \end{equation}
  in the sense of distributions. Moreover, the solution may be represented as
  \begin{equation}\label{eq:uRep}
    u^\eps_\alpha(x) = \int_{\R^3}\Elas_{ijkl}\GMat_{\alpha i,j}(x-y)(z^\eps_T)_{kl}(y)\dd y
    = \int_{\Sigma}\Elas_{ijkl}\GMat^\eps_{\alpha i,j}(x-y)b_k(s)\nu_l(s)\dd\Haus^2(s),
  \end{equation}
  where $\GMat$ is the \emph{elastic Green's function}, which, recalling
  \S\ref{sec:notation-conventions}, is the distributional solution of \eqref{eq:GMatEqn},
  and the function $\GMat^\eps:=\GMat*\varphi^\eps$ is a regularised version of $\GMat$, solving
  \eqref{eq:GMatEpsEqn}.
\end{lemma}

\begin{proof}
  The existence of a minimiser follows from the fact that by Theorem~5.25 in \cite{Dacorogna}, any
  quadratic function is quasiconvex if and only if it is rank--one convex, and in this case
  rank--one convexity of the integrand is straightforward to check, following from the assumption
  that $\Elas$ satisfies a Legendre--Hadamard condition.
  As a consequence, $\I$ is weakly lower semicontinuous on $\dot{H}^1(\R^3)$ and unique minimisers
  exist in this space; computing the Frechet derivative of $\I$ in the same space shows that the
  solution satisfies the equation \eqref{eq:ElasticResp} in the sense of distributions.

  Convolving the distributional equation \eqref{eq:GMatEqn} with $f_\beta$ and contracting the index $\beta$,
  we see that
  \begin{equation*}
    -\Elas_{\alpha jkl}\GMat_{\beta k} *f_{\beta,lj} = f_\alpha.
  \end{equation*}
  Setting $f_\alpha = \Elas_{\alpha pqr}(z^\eps_T)_{qr,p}$, we find that
  \begin{equation*}
    u_\alpha^\eps = \Elas_{\beta pqr}\GMat_{\beta \alpha} * (z^\eps_T)_{qr,p}.
  \end{equation*}
  Integrating by parts to move the derivative with respect to $x_j$ onto $\GMat$,
  then using the fact that $\GMat$ is a symmetric tensor, i.e. $\GMat_{ij}=\GMat_{ji}$ for any
  $i,j\in\{1,2,3\}$, we obtain the first equality in \eqref{eq:uRep}. The second equality
  follows by using Fubini's theorem to deduce that
  \begin{equation*}
    \GMat*\B(\varphi^\eps*\b(b\otimes\nu \Haus^2\Rstr\Sigma\b)\B) = \big(\GMat*\varphi^\eps\big)*\big(b\otimes\nu \Haus^2\Rstr\Sigma\big) = \GMat^\eps*\big(b\otimes\nu \Haus^2\Rstr_\Sigma\big).\qedhere
  \end{equation*}
\end{proof}

We remark that while Lemma~\ref{th:DistortionCharacterisation} establishes that the problem
\eqref{eq:SlipEnergy} has a unique solution $u^\eps$, it does not guarantee the positivity of
the energy $\mathcal{I}(z^\eps+Du^\eps)$, since we cannot say anything
about the sign of $\mathcal{I}(z^\eps)$ unless $z^\eps$ is a gradient.
We also note that under appropriate growth and quasiconvexity conditions, the existence of $u^\eps$
can be ensured if the stored energy density takes a more general nonlinear form; related ideas
are discussed in \cite{SZ12,MSZ14,MSZ15} in a two--dimensional setting.

\subsection{Energy}
A key feature of the linear theory we consider is that it allows the derivation of a
representation formula for the distortion due to a configuration of slip, which leads to
the following result, allowing us to provide an explicit integral formula for the internal energy
\eqref{eq:SlipEnergy}.

\begin{lemma}\label{th:EnergyIntegralFormula}
  If $T\in\Itwo(\R^3;\Latt)$, supported on $\Sigma$, with slip vector $b$ and normal field $\nu$,
  the energy $\E^\eps$ defined in \eqref{eq:SlipEnergy} may be represented
  \begin{equation}
    \label{eq:EnergySSForm}
    \begin{aligned}
      \E^\eps(T)
        &= \int_{\R^3}\int_{\Sigma\times\Sigma}\smfrac12\Elas_{abcd}
        \Alt_{bpl}\Elas_{ijkl}b_k(s)\GMat^\eps_{a i,jn}(x-s)\Alt_{pmn}\nu_m(s)\\
        &\hspace{30mm}\times\Alt_{dqh}\Elas_{efgh}b_g(t)\GMat^\eps_{c e,fs}(x-t)\Alt_{qrs}\nu_r(t)
        \,\dd(\Haus^2\otimes\Haus^2)(s,t)\dx
      \end{aligned}
    \end{equation}
  where $\Alt$ is the alternating tensor as defined in \S\ref{sec:notation-conventions}.
  
  Furthermore, \eqref{eq:EnergySSForm} depends only on $\partial T$, which entails that $\Phi^\eps:\Adm\to\R$ with
\begin{equation*}
  \Phi^\eps(S):=\E^\eps(T)\quad\text{for any }S\in\Adm\text{ where }S=\partial T
\end{equation*}
is well--defined, and if $S=\partial T$ is supported on $\Gamma$, with Burgers vector $b$ and
tangent field $\tau$, $\Phi^\eps(S)$ may be expressed as
\begin{equation}\label{eq:PhiDefinition}
  \begin{aligned}
  \Phi^\eps(S)
        &=\int_{\R^3}\int_{\Gamma\times \Gamma}\smfrac12\Elas_{abcd}
        \Alt_{bpl}\Elas_{ijkl}b_k(s)\GMat^\eps_{a i,j}(x-s)\tau_p(s)\\
        &\hspace{30mm}\times\Alt_{dqh}\Elas_{efgh}b_g(t)\GMat^\eps_{c e,f}(x-t)\tau_q(t)
        \,\dd(\Haus^1\otimes\Haus^1)(s,t)\dx.
      \end{aligned}
    \end{equation}
\end{lemma}

\begin{proof}
  Since $T$ is fixed during this proof, throughout, we write $z^\eps$ in place of
  $z^\eps_T$ to keep notation as concise as possible.
  Applying \eqref{eq:uRep}, we write the elastic distortion $\beta^\eps$  as
\begin{equation}\label{eq:betaRep1}
  \beta^{\eps}_{ab}= u^\eps_{a,b}+z^\eps_{ab}= \Elas_{ijkl}\GMat_{a i,j}*z^\eps_{kl,b}
  +z^\eps_{ab}.
\end{equation}
Using the definition of the elastic Green's function given in \eqref{eq:GMatEqn} with a change of indices,
integration by parts, and the major symmetry of the elasticity tensor,
\begin{equation*}
  z^\eps_{ab} = \Id_{ka}\delta_0 *z^\eps_{kb}
  = - \Elas_{klij}\GMat_{ai,jl}*z^\eps_{kb} 
  = - \Elas_{klij}\GMat_{ai,jl}*z^\eps_{kb,l}
  = - \Elas_{ijkl}\GMat_{ai,j}*z^\eps_{kb,l}.
\end{equation*}
Substituting this representation into \eqref{eq:betaRep1} in place of the latter term, we obtain
\begin{align*}
  \beta^{\eps}_{ab}(x) &= \int_{\R^3}\Elas_{ijkl}\B[\GMat_{a i,j}(x-y)z^\eps_{kl,b}(y)
                         -\GMat_{ai,j}(x-y)z^\eps_{kb,l}(y)\B]\dd y,\\
                       &=\int_{\R^3}\Elas_{ijkl}\GMat_{a i,j}(x-y)[\Id_{lm}\Id_{bn}
                         -\Id_{bm}\Id_{ln}] z^\eps_{km,n}(y)\dd y,\\
                       &=\int_{\R^3}\Alt_{plb}\Elas_{ijkl}\GMat_{a i,j}(x-y)\Alt_{pmn}z^\eps_{km,n}(y)\dd y,
\end{align*}
where we have used the elementary tensor identity $\Alt_{pmn}\Alt_{plb}=\Id_{lm}\Id_{bn}-\Id_{bm}\Id_{ln}$.
Using the definition of $z^\eps$ given in \eqref{eq:zepsDef}, we have
\begin{equation*}
  \beta^{\eps}_{ab}(x)
  =\int_{\R^3}\Alt_{plb}\Elas_{ijkl}b_k\GMat_{a i,j}(x-y)\int_{\Sigma}\Alt_{pmn}\nu_m(s)\varphi^\eps_{,n}(y-s)
  \dd\Haus^2(s)\,\dd y.
\end{equation*}
where $\tau$ is the tangent vector field on $\partial\Sigma$. If $x\notin\partial\Sigma$, Fubini's
theorem applies to the above integral representation, so using the definition of $\GMat^\eps$
given in \S\ref{sec:notation-conventions} and applying Stokes' Theorem, we find that
\begin{equation*}
  \beta^{\eps}_{ab}(x)
  =\int_{\Sigma}\Alt_{bpl}\Elas_{ijkl}b_k\GMat^\eps_{a i,jn}(x-s)\Alt_{pmn}\nu_m(s)\dd\Haus^2(s)
  =\int_{\partial\Sigma}\Alt_{bpl}\Elas_{ijkl}b_k\GMat^\eps_{a i,j}(x-s)\tau_p(s)\dd\Haus^1(s),
\end{equation*}
where $\tau:\partial\Sigma\to\Sph^2$ is the tangent vector field on $\partial\Sigma$. Substituting the former expression of $\beta^\eps$ into the definition of $\I$ stated in \eqref{eq:Energy} gives
\eqref{eq:EnergySSForm}. Similarly, substituting the latter expression into \eqref{eq:Energy}
allows us to directly deduce that $\Phi^\eps$ is well--defined and has the expression given in
\eqref{eq:PhiDefinition}.
\end{proof}

\subsection{Kernel representation}
\label{sec:repr-kern}
Inspecting the formulae for $\E^\eps$ and $\Phi^\eps$ given in Lemma~\ref{th:EnergyIntegralFormula},
we note that both expressions can be regarded as a convolution integral against an interaction
kernel; this form allows us to use ideas from \cite{BBS80} to prove the following result, which,
when combined with Lemma~\ref{th:DistortionCharacterisation} and
Lemma~\ref{th:EnergyIntegralFormula}, completes the proof of Theorem~\ref{th:Energy}.

\begin{lemma}\label{th:Kernels}
  If $T\in\Itwo(\R^3;\Latt)$ is supported on $\Sigma$, with slip vector $b$ and normal field $\nu$,
  and $S=\partial T$ is supported on $\Gamma$ with Burgers vector $b$ and tangent field $\tau$,
  the energy functionals $\E^\eps(T)$ and $\Phi^\eps(S)$ may be expressed as
  \begin{align*}
    \begin{split}
      \E^\eps(T) &= \int_{\Sigma\times\Sigma}\smfrac12\Jker_{abcd}(s-t)b_a(s)\nu_b(s)b_c(t)\nu_d(t)\dd(\Haus^2\otimes\Haus^2)(s,t)
    \end{split}\\
    \begin{split}
      \Phi^\eps(S)&= \int_{\Gamma\times\Gamma}\,\smfrac12\Kker_{abcd}(s-t)b_a(s)\tau_b(s)b_c(t)\tau_d(t)
      \dd(\Haus^1\otimes\Haus^1)(s,t),
    \end{split}
  \end{align*}
  where the kernels $\Jker,\Kker:\R^3\to\R^{3\times3\times3\times3}$ are defined to be
  \begin{align*}
    \Jker_{kmgr}(s)
    &:=\int_{\R^3}\Elas_{abcd}\Alt_{bpl}\Elas_{ijkl}\GMat^\eps_{a i,jn}(x-s)\Alt_{pmn}
      \Alt_{dqh}\Elas_{efgh}\GMat^\eps_{c e,fs}(x)\Alt_{qrs}\dx,\\
    \Kker_{kpgq}(s)
    &:=\int_{\R^3}\Elas_{abcd}\Alt_{bpl}\Elas_{ijkl}\GMat^\eps_{a i,j}(x-s)\Alt_{dqh}\Elas_{efgh}
      \GMat^\eps_{c e,f}(x)\dx,
  \end{align*}
  and satisfy the following properties.
  \begin{enumerate}
  \item $\Jker_{abcd}(s)=\Jker_{cdab}(s)=\Jker_{abcd}(-s)$ and
    $\Kker_{abcd}(s)=\Kker_{cdab}(s)=\Kker_{abcd}(-s)$ for any $s\in\R^3$;
  \item $\Jker$ and $\Kker$ are smooth.
  \item For any $m\in\N$, $0\leq j\leq m$, and vectors $v_1,\ldots,v_j\in\Sph^2$,
    there exists a constant $C_{m,j}$ such that for all $s\in\R^3$,
    \begin{equation*}
      \B|D^m \Kker(s):\B[v_1,\ldots, v_j,\frac{s}{|s|},\ldots,\frac{s}{|s|}\B]\B|
      \leq \frac{C_{m,j}}{\sqrt{\eps^{2m+2}+\eps^{2j}|s|^{2(m-j)+2}}}.
    \end{equation*}
  \end{enumerate}
\end{lemma}

\begin{proof}
  We divide the proof into a series of steps, corresponding to each of the assertions made.
  \medskip

  \noindent\emph{Kernel representation.}
  The existence of the kernels is a straightforward consequence of applying Fubini's theorem to
  the expressions \eqref{eq:EnergySSForm} and \eqref{eq:PhiDefinition}; $\Jker(s)$ and $\Kker(s)$
  are finite for
  any $s\in\R^3$ since the elastic Green's function satisfies the standard properties that
  $|\GMat_{ij,k}(x)|\lesssim |x|^{-2}$ and $|\GMat_{ij,kl}(x)|\lesssim |x|^{-3}$,
  and therefore there exist constants $C_\eps$ such that
  \begin{equation*}
    \b|\GMat^\eps_{ij,k}(x)\b|\leq C_\eps\min\b\{ 1, |x|^{-2}\b\}\quad\text{and}\quad\b|\GMat^\eps_{ij,kl}(x)\b|\leq C_\eps
    \min\b\{ 1, |x|^{-3}\b\},
  \end{equation*}
  and it follows that $\GMat^\eps_{ij,k}$ and $\GMat^\eps_{ij,kl}$ are in $L^2(\R^3)$, and hence the
  integrals in \eqref{eq:KerDefs} converge for any $s\in\R^3$.\medskip

  \noindent\emph{Fourier characterisation of kernels.}
  To prove that the kernels $\Jker$ and $\Kker$ satisfy the stated properties, we use a
  characterisation via the Fourier transform.
  Applying the Fourier transform to the definition of $\GMat^\eps$ given in
  \S\ref{sec:notation-conventions}, we obtain
\begin{equation*}
  -\Elas_{abcd}\widehat{\GMat^\eps_{ec,db}}(k) = \Elas_{abcd}k_bk_d\widehat{\GMat^\eps_{ec}}(k) = \Id_{ae}\widehat{\varphi^\eps}(k).
\end{equation*}
Define the $2$--tensor $\Dsf(k)_{ac}: = \Elas_{abcd}k_bk_d$, and its algebraic inverse
$\Dsf(k)^{-1}$, satisfying the relation $\Dsf(k)^{-1}_{ab}\Dsf(k)_{bc}=\Id_{ac}$. $\Dsf(k)^{-1}$ is
well--defined for $k\neq0$, since the Legendre--Hadamard condition on $\Elas$ entails that
$\Dsf(k)$ is strictly positive definite in this case, and it follows that
\begin{equation*}
  \widehat{\GMat^\eps_{ec}}(k) = \Dsf(k)^{-1}_{ec}\widehat{\varphi^\eps}(k)\qquad
  \text{and}\qquad\widehat{\GMat^\eps_{ab,c}}(k)= -ik_c\Dsf(k)^{-1}_{ab}\widehat{\varphi^\eps}(k).
\end{equation*}
Since $\varphi^\eps$ was assumed to be smooth and rapidly--decreasing, the same holds for
$\widehat{\varphi^\eps}$. Moreover, $\Dsf(k)^{-1}_{ab}k_c$ is $-1$--homogeneous in $k$, i.e.
\begin{equation}\label{eq:-1Homogeneous}
  \Dsf(\lambda k)^{-1}_{ab}\lambda k_c = \lambda^{-1} \Dsf(k)^{-1}_{ab} k_c,\quad\text{for any }\lambda\neq 0.
\end{equation}
Using this observation and applying Plancherel's theorem to the definitions in \eqref{eq:KerDefs},
\begin{align*}
  \Jker_{kmgr}(s)
  &= \int_{\R^3}\Elas_{abcd}\Alt_{bpl}\Elas_{ijkl}\Alt_{pmn}
    \Alt_{dqh}\Elas_{efgh}\Alt_{qrs}
    \widehat{\GMat^\eps_{ai,jn}}(k)\overline{\widehat{\GMat^\eps_{ce,fs}}(k)}\mathrm{e}^{-ik\cdot s}\dd k\\
  &= \int_{\R^3}\Elas_{abcd}\Alt_{bpl}\Elas_{ijkl}\Alt_{pmn}
    \Alt_{dqh}\Elas_{efgh}\Alt_{qrs}k_jk_nk_fk_s
    \Dsf(k)_{ai}^{-1}\Dsf(k)_{ce}^{-1}\b|\widehat{\varphi^\eps}(k)\b|^2\mathrm{e}^{-ik\cdot s}\dd k,\\
  \Kker_{abcd}(s)
  &= \int_{\R^3}\Elas_{efgh}\Elas_{aijk}\Elas_{clmn}\Alt_{fib}\Alt_{hld}
    \widehat{\GMat^\eps_{ej,k}}(k)\overline{\widehat{\GMat^\eps_{gm,n}}(k)}\mathrm{e}^{-ik\cdot s}\dd k\\
  &= -\int_{\R^3}\Elas_{efgh}\Elas_{aijk}\Elas_{clmn}\Alt_{fib}\Alt_{hld}
    k_kk_n\Dsf(k)^{-1}_{ej}\Dsf(k)^{-1}_{gm}\b|\widehat{\varphi^\eps}(k)\b|^2
    \mathrm{e}^{-ik\cdot s}\dd k.
\end{align*}
As $\varphi^\eps$ is assumed to be radially symmetric, it follows that
$\widehat{\varphi^\eps}$ is also radially--symmetric, and therefore
$k_kk_n\Dsf(k)_{ej}^{-1}\Dsf(k)_{gm}^{-1}|\widehat{\varphi^\eps}(k)|^2$ is even in $k$.
Using the latter observation, and setting $r = |k|$ and decomposing $k = r z$ for some
$z\in\mathbb{S}^2=\{z\in\R^3\sep |z|=1\}$, we transform to polar coordinates, and
use \eqref{eq:-1Homogeneous} and the evenness of $\widehat{\varphi^\eps}$ to obtain
\begin{align*}
  \Kker_{abcd}(s)
  &= -\int_{\R^3}\Elas_{efgh}\Elas_{aijk}\Elas_{clmn}\Alt_{fib}\Alt_{hld}
    z_kz_n\Dsf(z)^{-1}_{ej}\Dsf(z)^{-1}_{gm}r^{-2}\b|\widehat{\varphi^\eps}(rz)\b|^2
    \cos(rz\cdot s)\dd k\\
  &=- \int_{\mathbb{S}^2}\Elas_{efgh}\Elas_{aijk}\Elas_{clmn}\Alt_{fib}\Alt_{hld}
    z_kz_n\Dsf(z)^{-1}_{ej}\Dsf(z)^{-1}_{gm}
    \bg(\frac12\int_{-\infty}^{+\infty}\b|\widehat{\varphi^\eps}
    (rz)\b|^2\mathrm{e}^{ir z\cdot s}\dd r\bg)\,\dd\Haus^2(z).
\end{align*}
Standard properties of the Fourier transform imply that $\widehat{\varphi^\eps}(k)
= \widehat{\varphi^1}(\eps k)$, so applying this relation and changing variable, we find
\begin{equation*}
  \int_{-\infty}^{+\infty}\b|\widehat{\varphi^\eps}(rz)\b|^2\mathrm{e}^{ir z\cdot s}\dd r
  =\frac{1}{\eps}\int_{-\infty}^{+\infty}\b|\widehat{\varphi^1}(rz)\b|^2\mathrm{e}^{ir z\cdot s/\eps}\dd r.
\end{equation*}
Now, for any $\eps>0$, we define $\eta^\eps:\R\to\R$ to be
\begin{equation*}
  \eta^\eps(t) := \frac{1}{\eps}\int_{-\infty}^{+\infty}\b|\widehat{\varphi^1}(r e_1)\b|^2\mathrm{e}^{ir t/\eps}
  \dd r = \eta^1(t/\eps)/\eps.
\end{equation*}
It is straightforward to show that this function is rapidly--decreasing, a property it inherits
from $\varphi^\eps$.
In summary, we have shown that
\begin{equation}\label{eq:KFourierRep}
  \Kker_{abcd}(s) = -\int_{\mathbb{S}^2}\smfrac12\Elas_{efgh}\Elas_{aijk}\Elas_{clmn}\Alt_{fib}\Alt_{hld}
  z_kz_n\Dsf(z)^{-1}_{ej}\Dsf(z)^{-1}_{gm}\eta^\eps(z\cdot s)\,\dd\Haus^2(z).
\end{equation}
Performing a similar computation for $\Jker$, we obtain
\begin{align*}
  \Jker_{kmgr}(s)
  &= \int_{\Sph^2}\smfrac12\Elas_{abcd}\Elas_{ijkl}\Elas_{efgh}\Alt_{bpl}\Alt_{pmn}
    \Alt_{dqh}\Alt_{qrs}z_jz_nz_fz_s
    \Dsf(z)_{ai}^{-1}\Dsf(z)_{ce}^{-1}\\
  &\hspace{5cm}\times\bg(\int_{-\infty}^\infty r^2\b|\widehat{\varphi^\eps}(rz)\b|^2\mathrm{e}^{irz\cdot s}\dd r \bg)\dd \Haus^2(z).
\end{align*}
Considering the inner integral and performing a change of variable,
\begin{equation*}
  \int_{-\infty}^\infty r^2\b|\widehat{\eta^\eps}(rz)\b|^2\mathrm{e}^{ir t}\dd r = \frac{1}{\eps^3}\int_{-\infty}^{+\infty}r^2\b|\widehat{\eta^1}(r e_1)\b|^2\mathrm{e}^{irt/\eps}
  \dd r = -(\eta^\eps)''(t);
\end{equation*}
hence
\begin{equation*}
  \Jker_{kmgr}(s)= \int_{\Sph^2}\smfrac12\Elas_{abcd}\Elas_{ijkl}\Elas_{efgh}\Alt_{bpl}\Alt_{pmn}
    \Alt_{dqh}\Alt_{qrs}z_jz_nz_fz_s\Dsf(z)_{ai}^{-1}\Dsf(z)_{ce}^{-1}(\eta^\eps)''(z\cdot s)\,\dd\Haus^2(z).
\end{equation*}
By applying a series of tensor identities, this representation can be reduced to
\begin{equation}
 \label{eq:JFourierRep}
 \Jker_{kmgr}(s)= \int_{\Sph^2}\smfrac12
 \b[\Elas_{kmgr}-\Elas_{abgr}\Dsf(z)_{ai}^{-1}\Elas_{ijkm}z_bz_j\b]
 (\eta^\eps)''(z\cdot s)\,\dd\Haus^2(z).
\end{equation}\medskip

\noindent\emph{Kernel properties.}
The symmetry and smoothness properties asserted in (1) and (2)
follow directly from the representations \eqref{eq:KFourierRep} and \eqref{eq:JFourierRep},
noting that $\eta^\eps$ is smooth by construction.

Next, we note that as $\eta^1\in C^\infty(\R)$ is a rapidly--decreasing function, there exist constants $C_m>0$ for $m\in\N$, independent of
$\eps$, such that
\begin{equation}\label{eq:PsiEpsBounds}
  \b|(\eta^\eps)^{(m)}(r)\b| \leq \frac{C_m}{\eps^{m+1}}\quad\text{for all }r\in\R.
\end{equation}
Moreover, since $\Elas$ satisfies a Legendre--Hadamard condition and $\Dsf(k)$ is strictly positive definite,
it follows that there exist $M$ and $M'$ such that for all $z\in\Sph^2$
\begin{equation}\label{eq:SphBounds}
  \begin{gathered}
    \left|\Elas_{efgh}\Elas_{aijk}\Elas_{clmn}\Alt_{fib}\Alt_{hld}
      z_kz_n\Dsf(z)^{-1}_{ej}\Dsf(z)^{-1}_{gm}\right|\leq M\quad\text{and}\\
    \left|\Elas_{kmgr}-\Elas_{abgr}\Dsf(z)_{ai}^{-1}\Elas_{ijkm}z_bz_j\right|\leq M'.
  \end{gathered}
\end{equation}
Applying the bounds \eqref{eq:PsiEpsBounds} and \eqref{eq:SphBounds} to the representations \eqref{eq:KFourierRep}
and \eqref{eq:JFourierRep}, we obtain
  \begin{equation}\label{eq:UniformDerivBound}
    \b|D^m\Kker_{abcd}(s)\b|\leq \frac{4\pi C_mM}{\eps^{m+1}}\quad\text{and}\quad
    \b|D^m\Jker_{abcd}(s)\b|\leq \frac{4\pi C_{m+2}M}{\eps^{m+3}}
    \quad\text{for all }s\in\R^3\text{ and }m\in\N.
  \end{equation}
  Further, taking derivatives of \eqref{eq:KFourierRep}, applying the resulting multilinear
  operator to the collection of vectors $v_1,\ldots,v_j,s\ldots,s$, where $|v_i|=1$, and using
  \eqref{eq:SphBounds} once more,
  we find that
  \begin{equation*}
    \B|D^m \Kker(s):\b[v_1,\ldots,v_j,s,\ldots,s\b]\B|
    \leq M \int_{\Sph^2}|z\cdot s|^{m-j} \b|(\eta^\eps)^{(m)}(z\cdot s)\b|\dd z.
  \end{equation*}
  Expressing this upper bound using polar coordinates on $\Sph^2$ with inclination $\theta$
  measured relative to an axis parallel to $s$, and subsequently changing variable to
  $t=\frac{|s|}{\eps}\cos\theta$, we have
  \begin{equation}\label{eq:DecayDerivBound}
    \begin{aligned}
      \B|D^m \Kker(s)\b[v_1,\ldots,v_j,s,\ldots,s\b]\B|
      &\leq 2\pi M \int_0^\pi \bg|(\eta^1)^{(m)}\left(\frac{|s|\cos\theta}{\eps}\right)\bg|
      \frac{|s|^{m-j}|\!\cos\theta|^{m-j}\sin\theta}{\eps^{m+1}}\dd\theta\\
      &= \frac{2\pi M}{\eps^{j}|s|} \int_{-|s|/\eps}^{|s|/\eps}\hspace{-2mm}|t|^{m-j}\b|(\eta^1)^{(m)}(t)\b|\dt\\
      &\leq\frac{2\pi M}{\eps^{j}|s|} \int_{-\infty}^\infty|t|^{m-j}\b|(\eta^1)^{(m)}(t)\b|\dt,
    \end{aligned}
  \end{equation}
  where the integral in this upper bound is finite since $\eta^1$ is a rapidly--decreasing
  function. Dividing by $|s|^{m-j}$, and combining with
  \eqref{eq:UniformDerivBound} completes the proof of assertion (3).
\end{proof}

\noindent
We make the following remarks concerning the Fourier representations of the kernels given
in formulae \eqref{eq:KFourierRep} and \eqref{eq:JFourierRep}:
\begin{itemize}
\item In the case where $\varphi^\eps$ is a Gaussian, as in the example provided in
  \eqref{eq:Gaussian}, we may explicitly compute $\eta^\eps$ as used in the proof above,
  giving
\begin{equation}\label{eq:Explicitpsi}
  \eta^\eps(t)
  = \frac{8\pi^{7/2}}{\eps}\exp\B(-\frac{t^2}{4\eps^2}\B).
\end{equation}
More generally, for the purpose of computation we may choose $\varphi^\eps$ in order to obtain
a convenient expression for $\eta^\eps$.
\item Combining a convenient choice for $\eta^\eps$ with the representations
  \eqref{eq:KFourierRep} and \eqref{eq:JFourierRep} suggests that the kernels
  $\Jker$ and $\Kker$ may be efficiently computed numerically, since these expressions require
  integration of a smooth function over the unit sphere, and this can be accurately approximated
  in practice with relatively few quadrature points. Moreover, these expressions are amenable
  to asymptotic analysis in the case where $|s|\gg\eps$, which should allow for the implementation
  of explicit expressions to speed--up computation.
\end{itemize}

\section{Deforming dislocations and the Peach--Koehler force}
\label{sec:deform-disl-peach}
Theorem~\ref{th:Energy} established a representation of the elastic energy
induced in a material due to the presence of dislocations. In this section, we prove
Theorem~\ref{th:Force}, computing the configurational or Peach--Koehler force induced on a
dislocation configuration, and demonstrating its properties.


\subsection{The Peach--Koehler force}\label{sec:peach-koehler-force}

The first step towards proving Theorem~\ref{th:Force} is to establish the expressions \eqref{eq:PKForce} and \eqref{eq:PKForceSurface}.

\begin{lemma}
  \label{th:PKForce}
  If $S\in\Adm$, the inner variation of $\Phi^\eps$, defined in \eqref{eq:PhiDefinition},
  is given by
  \begin{equation*}
    \begin{gathered}
      \<D\Phi^\eps(S),g\> = -\int_\Gamma \PK(S)\cdot g\,\dd\Haus^1(s),\quad\text{where}
      \quad \PK(s,S):=G(s,S)\wedge\tau(s)\\
      \text{and}\quad G_k(s,S):=\int_\Gamma\Alt_{klm}\Kker_{alcd,m}(s-t)
      b_a(s)b_c(t)\tau_d(t)\dd\Haus^1(t).
    \end{gathered}
  \end{equation*}
  Moreover, if $S = \partial T$, where $T\in\Itwo(\R^3;\Latt)$ is supported on $\Sigma$ with
  slip vector $b$ and normal field
  $\nu$, $G$ can alternatively be written
  \begin{equation*}
    G_k(s,S) = \int_\Sigma\Alt_{def}\Alt_{klm}\Kker_{alcd,mb}(s-t)b_a(s)b_c(t)\nu_f(t)\dd\Haus^2(t).
  \end{equation*}
\end{lemma}

\begin{proof}
Given $g\in C^1(\R^3;\R^3)$, we set $h^\del:=\id+\del g$. Pushing forward, we find that
\begin{multline*}
  \Phi^\eps(h^\del_\#S)
  = \int_{\Gamma\times\Gamma}\!\smfrac12\Kker_{abcd}\b(s-t+\del (g(s)- g(t))\b)
  b_a(s)\b(\tau_b(s)+\del \nabla_\tau g_b(s)\b)b_c(t)\\
  \times\b(\tau_d(t)+\del\nabla_\tau g_d(t)\b)\dd(\Haus^1\otimes\Haus^1)(s,t).
\end{multline*}
Applying the definition \eqref{eq:InnerVariation}, we differentiate and set $\delta=0$,
we find
\begin{align*}
  \<D\Phi^\eps(S),g\>
  &= \int_{\Gamma\times\Gamma}\B[\smfrac12\Kker_{abcd,e}(s-t)(g_e(s)-g_e(t))b_a(s)\tau_b(s)b_c(t)\tau_d(t)\\
  &\hspace{2cm}+\smfrac12\Kker_{abcd}(s-t)b_a(s)\nabla_\tau g_b(s)b_c(t)\tau_d(t)\\
  &\hspace{3cm}+\smfrac12\Kker_{abcd}(s-t)b_a(s)\tau_b(s)b_c(t)\nabla_\tau g_d(t)\B]\dd(\Haus^1\otimes\Haus^1)(s,t).
\end{align*}
Applying the symmetries of $\Kker$ asserted in Lemma~\ref{th:Kernels}, this formula reduces to
\begin{align*}
  \<D\Phi^\eps(S),g\>
  &= \int_{\Gamma\times\Gamma}\B[\Kker_{abcd,e}(s-t)g_e(s)b_a(s)\tau_b(s)b_c(t)\tau_d(t)\\
  &\hspace{2cm}+\Kker_{abcd}(s-t)b_a(s)\nabla_\tau g_b(s)b_c(t)\tau_d(t)\B]
    \dd(\Haus^1\otimes\Haus^1)(s,t).
\end{align*}
Since $S\in\Adm$ satisfies $\partial S =0$, i.e. $S$ is formed of closed loops,
we may integrate by parts in the variable $s$, passing a derivative from $g$ onto $\Kker$ in the
second term, which yields
\begin{multline*}
  \<D\Phi^\eps(S),g\>
  = \int_{\Gamma\times\Gamma}\B[\Kker_{abcd,e}(s-t)g_e(s)b_a(s)\tau_b(s)b_c(t)\tau_d(t)\\
  -\Kker_{abcd,e}(s-t)b_a(s)g_b(s)\tau_e(s)b_c(t)\tau_d(t)\B]
  \dd(\Haus^1\otimes\Haus^1)(s,t)
\end{multline*}
Now, applying the tensor identity $\Alt_{ijk}\Alt_{klm}=\Id_{il}\Id_{jm}-\Id_{im}\Id_{jl}$,
and the definition of $G(s,S)$, we obtain the first result. To obtain the latter expression,
we simply apply Stokes' Theorem.
\end{proof}

\noindent
In view of this result, we make two remarks:
\begin{itemize}
\item Without additional regularity assumptions on $S$, we note that $\PK$ is generically
  only in $L^\infty(S)$, since the tangent field $\tau$ on a Lipschitz curve need not be continuous.
\item More generally, the fact that $\PK$ is the product of a smooth kernel with components of
  $b$ (which is locally constant on $S$) and the tangent field $\tau$, entails that the regularity
  of $\PK$ at a point is dictated by the regularity of the tangent field $\tau$ at the same
  point. This point is one we will return to in \S\ref{sec:evol-probl-exist} when formulating a
  dynamical theory.
\item If $g$ and $S$ are assumed to be more regular, it is possible to compute higher--order
  variations of the energy in a similar way. However, it should be noted that some care is
  required if variations are made in different directions, since the order in which variations are
  taken will matter in general.
\end{itemize}

\subsection{Bounds on the Peach--Koehler force}
The second crucial step in proving Theorem~\ref{th:Force} is to establish \eqref{eq:PKLInfty},
which is encoded in the following result.

\begin{lemma}
  If $S\in\Adm$ is an admissible dislocation configuration with Burgers vector $b$, then the
  Peach--Koehler force satisfies the uniform bound
  \begin{equation*}
    \b\|\PK(S)\b\|_{L^\infty}\leq
    \frac C\eps\|b\|_{L^\infty}\MRatio(S)\log\bg|1+\frac{2\,\Mass(S)}{\eps\,\MRatio(S)}\bg|,
  \end{equation*}
  where $C>0$ is a coefficient independent of $S$ and $\eps$, $\|b\|_{L^\infty}$ is the maximum
  Burgers vector, and the mass $\Mass(S)$ and mass ratio $\MRatio(S)$ were respectively defined
  in \eqref{eq:Mass} and \eqref{eq:MassRatio}.
\end{lemma}

\begin{proof}
  As a consequence of assertion (3) in Lemma~\ref{th:Kernels}, we have that
  \begin{equation*}
    \b|D\Kker(s)\b|\leq \frac{C}{\eps\sqrt{\eps^2+|s|^2}},
  \end{equation*}
  with $C$ independent of $\eps$, and therefore the expression given for $\PK$ in \eqref{eq:PKForce} directly implies that
  \begin{equation*}
    \b|\PK(s,S)\b| \leq \|b\|_{L^\infty}\int_{\Gamma}\b|D \Kker(s-t)\b||b(t)|\dd\Haus^1(t)
    \leq
    \frac{C}\eps\|b\|_{L^\infty}\int_\Gamma\frac{|b(t)|}{\sqrt{\eps^2+|s-t|^2}}\dd\Haus^1(t),
  \end{equation*}
  where $\Gamma$ is the support of $S$.

  This upper bound may now be recast in the following way: Define
  the function $\mu:\R^+\to\R^+$ to be
  \begin{equation*}
    \mu(r):=\Mass\b(S\Rstr\overline{B_r(s)}\b).
  \end{equation*}
  This function is clearly monotonically increasing in $r$, satisfies $\mu(0)=0$, and since
  $S\in\Adm$ is compactly--supported, there must exist $R\geq 0$ for which
  \begin{equation}\label{eq:MaxMass}
    \mu(r) = \Mass(S)\quad\text{whenever }r\geq R.
  \end{equation}
  As a consequence of these facts, $\mu$ is a function of bounded variation
  (see \S3.2 of \cite{AFP00}), and has a weak derivative, $\mu'$, which may in general be a
  measure. Using the definition of $\mu(r)$, we have
  \begin{equation*}
   \int_\Gamma \frac{|b(t)|}{\sqrt{\eps^2+|s-t|^2}}\dd\Haus^1(t)
    = \int_0^\infty \frac{1}{\sqrt{\eps^2+r^2}}\dd \mu'(r).
  \end{equation*}
  Now, to proceed, we define the inverse of $\mu$,
  \begin{equation*}
    \rho(m):=\inf_{r\geq 0}\{\mu(r)\leq m\}.
  \end{equation*}
  We note that as a consequence of the observation made in \eqref{eq:MaxMass},
  $\rho(m)=+\infty$ for $m> \Mass(S)$.
  Changing variable by setting $\rho(m)=r$, and since by definition, $m=\mu(\rho(m))$,
  so that $1=\mu'(\rho(m))\rho'(m)$, we have
  \begin{equation*}
    \int_0^\infty \frac{1}{\sqrt{\eps^2+r^2}}\dd\mu'(r)
    =\int_0^{\Mass(S)} \frac{1}{\sqrt{\eps^2+\rho(m)^2}}\dd m.
  \end{equation*}
  To estimate this integral, we use the definition of $\MRatio(S)$ given in
  \eqref{eq:MassRatio} to find that
  \begin{equation*}
    \mu(r)\leq \MRatio(S)r\quad\text{for all }r\geq0,\qquad
  \text{and hence}\qquad
    \rho(m) \geq \frac{m}{\MRatio(S)}\quad\text{for all }m\geq 0.
  \end{equation*}
  It follows that the latter integral may be bounded above by
  \begin{multline*}
    \int_0^{\Mass(S)} \frac{1 }{\sqrt{\eps^2+\rho(m)^2}} \dd m
    \leq\int_0^{\Mass(S)} \frac{1 }{\sqrt{\eps^2+m^2/\MRatio(S)^2}} \dd m\\
    =\MRatio(S) \log\Bg|\frac{\Mass(S)}{\eps\,\MRatio(S)}+\sqrt{1 + \frac{\Mass(S)^2}{\eps^2\MRatio(S)^2}}\Bg|
    \leq\MRatio(S) \log\bg|1+\frac{2\,\Mass(S)}{\eps\,\MRatio(S)}\bg|,
  \end{multline*}
  which directly entails the stated result.
\end{proof}

As discussed in \S\ref{th:Force}, this estimate, while better than simply using the fact that
$D\Kker$ is globally bounded,
does not take into account much detail of the geometric structure, nor the fact that the
Peach--Koehler force may be cast as either a line or surface integral (see the result of
Lemma~\ref{th:PKForce}). It may be of interest for future applications to improve this bound
in order to take better account of additional geometric features of a given dislocation
configuration.

As a direct consequence of \eqref{eq:PKLInfty}, we obtain
the following $L^2$ bound directly via H\"older's inequality.

\begin{corollary}
  \label{th:PKL2}
  We have the following bound on the Peach--Koehler force:
  \begin{equation}
    \label{eq:PKL2Bound}
    \b\|\PK\b\|_{L^2}\leq \frac C\eps\|b\|_{L^\infty}\Mass(S)^{1/2}\MRatio(S)
    \log\bg|1+\frac{2\,\Mass(S)}{\eps\,\MRatio(S)}\bg|.
  \end{equation}
\end{corollary}
\medskip

\noindent
Both $L^\infty$ and $L^2$ bounds, \eqref{eq:PKLInfty} and \eqref{eq:PKL2Bound}, will be important
for the proof of Theorem~\ref{th:WellPosedness}.

\subsection{Continuity of the Peach--Koehler force}
To complete the proof of Theorem~\ref{th:Force}, we establish \eqref{eq:PKContinuity}.

\begin{lemma}
  \label{th:PKContinuity}
  If $g:S\to\R^3$ is a Lipschitz map and $F:=\id+g$, then
  \begin{equation*}
    \b\|F^\#\PK(F_\#S)-\PK(S)\b\|_{L^\infty}\leq \b(1+C\Mass(S)\b)\|\nabla_\tau g\|_{L^\infty}+C\Mass(S)\|g\|_{L^\infty},
  \end{equation*}
  where $C$ is a constant independent of $g$ and $S$.
\end{lemma}

\begin{proof}
  We recall that
  \begin{equation}\label{eq:PKCont1}
    \PK(s,S)=\bg(\int_\Gamma\Alt_{klm}\Kker_{alcd,m}(s-t)
    b_a(s)b_c(t)\tau_d(t)\dd\Haus^1(t)\bg)\wedge\tau(s),
  \end{equation}
  and so
  \begin{equation}\label{eq:PKCont2}
    \PK\b(F(s),F_\#S\b) = \bg(\int_{\Gamma}\Alt_{klm}\Kker_{alcd,m}\b(F(s)-F(t)\b)
    b_a(s)b_c(t)DF(t)[\tau(t)]_d\dd\Haus^1(t)\bg)\wedge DF(s)[\tau(s)].
  \end{equation}
  Taking the difference between these formulae, applying the triangle inequality
  and using the fact that
  \begin{equation*}
    \b|DF(s)[\tau(s)]-\tau(s)\b|\leq \|\nabla_\tau g\|_{L^\infty}
  \end{equation*}
  and since by assertion (3) of Lemma~\ref{th:Kernels}, $D^2\Kker$ is uniformly bounded, we may
  Taylor expand to obtain
  \begin{align*}
    \b|D\Kker\b(s-t+g(s)-g(t)\b)-D\Kker(s-t)\b|
    &=\B|D^2\Kker\B(s-t+\theta\b(g(s)-g(t)\b)\B)\b[g(s)-g(t)\b]\B|\\
    &\leq C\|g\|_{L^\infty}.
  \end{align*}
  Taking the difference between \eqref{eq:PKCont1} and \eqref{eq:PKCont2}, and applying the
  triangle inequality along with the latter estimates, we directly deduce the result.
\end{proof}

\section{Evolution problem and existence results}
\label{sec:evol-probl-exist}
We now prove Theorem~\ref{th:WellPosedness}. Our basic strategy for doing so
follows a fairly standard scheme. We carry out the following steps:
\begin{enumerate}
\item Construct a family of approximate solutions.
\item Derive bounds on the approximate solutions which are independent of the approximation.
\item Use these bounds and a compactness result to extract a convergent approximating sequence.
\item Prove that the approximating sequence satisfies \eqref{eq:DDD} in the limit, and verify
  the solution is unique.
\end{enumerate}
Each of these steps is carried out in turn over the course of the following sections.

\subsection{Approximation scheme}
In order to prove existence of a dynamical evolution, we set up an approximation scheme, which
may be viewed as an \emph{explicit Euler scheme} for the gradient flow dynamics.
Fixing $T>0$ and a sequence of times
\begin{equation*}
  0=t^0<t^1<\ldots<t^K=T,\quad\text{and set}\quad \delta t^i:= t^{i+1}-t^i,
\end{equation*}
for each $i=0,\ldots,K-1$, we will say that $(U,\{v^i\}_{i=0}^{K-1})$ form an \emph{approximate solution}
to DDD if $U(t) = U(t^i)^\#(\id+(t-t^i)v^i)$ for all $t\in(t^i,t^{i+1}]$ and $i=0,\ldots,K-1$, and
$v^i$ satisfies 
\begin{equation}\label{eq:Timestep}
    v^i\in\argmin_{v\in H^1(S^{i};\R^3)} \Psi(S^i,v)-\b(\PK(S^i),v\b)_{L^2(S^i)},
    \quad\text{where}\quad
  S^{i+1}:= (\id + \delta t^i\,v^i)_\#S^i
\end{equation}
for each $i=0,\ldots,K-1$.

We will prove that each of the minimisation problems \eqref{eq:Timestep} is
well--posed, and given sufficient regularity of $S^i$, $v^i$ is regular.
The following lemma encodes the first of these results.

\begin{lemma}
  \label{th:vExistence}
  If $S^i\in\Adm$, then then the minimisation problem in
  \eqref{eq:Timestep} has a unique solution $v^i\in H^1(S^i;\R^3)$, which satisfies
  \begin{equation}\label{eq:EL}
    \partial_v\Psi(S^i,v^i) \ni\PK(S^i)
    \quad\text{in }H^1(S^i,\R^3)^*
  \end{equation}
  and the bound
  \begin{equation}\label{eq:APVelBound}
        \|v^i\|_{H^1}
        \leq \frac {C\,\Mass(S^i)^{1/2}\MRatio(S^i)\|b\|_{L^\infty}}{\eps\min(\alpha,\beta)}
    \log\bg|1+\frac{2\,\Mass(S^i)}{\eps\,\MRatio(S^i)}\bg|.
  \end{equation}
\end{lemma}

\begin{proof}
  First, setting $v=0$ demonstrates that the functional which we seek to
  minimise is finite for some $v\in\HH^1(S^i;\R^3)$.
  Assumptions {\bf (C$_1$)} and {\bf(G)} entail that
  \begin{equation}\label{eq:Psi*LowerBound}
    \Psi(S^i,v) \geq \int_S \tfrac12\alpha
    |\nabla_\tau v|^2+\tfrac12\beta|v|^2\dd\Haus^1\geq\tfrac12\gamma\|v\|_{H^1}^2
    \qquad\text{for any }v\in H^1(S^i,\R^3),
  \end{equation}
  where $\gamma = \min(\alpha,\beta)$.
    Using the bound for $\PK$ derived in Corollary~\ref{th:PKL2}, we also have
  \begin{equation}\label{eq:BasicPKboundL2}
    \bg|\int_{S^i}\PK\cdot v\,\dd\Haus^1\bg|
    \leq \|v\|_{L^2}\frac C\eps\|b\|_{L^\infty}\Mass(S)^{1/2}\MRatio(S)
    \log\bg|1+\frac{2\,\Mass(S)}{\eps\,\MRatio(S)}\bg|.
  \end{equation}
  Combining \eqref{eq:Psi*LowerBound} and \eqref{eq:BasicPKboundL2} and using Young's
  inequality in the usual way, we find that
  \begin{equation*}
    \Psi(S^i,v)-(\PK(S^i),v)_{L^2}
    \geq \tfrac14\gamma\|v\|_{H^1}^2-
    \frac{C^2\|b\|_{L^\infty}^2\Mass(S)\MRatio(S)^2}{2\gamma\eps^2}
    \log\bg|1+\frac{2\,\Mass(S)}{\eps\,\MRatio(S)}\bg|^2
  \end{equation*}
  which implies that the functional which we seek to minimise is coercive.
  As $\Psi$ is strictly convex, and as $D\Phi^\eps(S^i)$ is a bounded linear functional and is
  therefore also convex,
  it follows that the map $v\mapsto\Psi(S^i,v)-(\PK(S^i),v)_{L^2}$ is weakly lower semicontinuous.
  A standard application of the Direct Method of the Calculus of Variations therefore implies
  existence of $v^i$, and strict convexity entails that $v^i$ is unique.

  To establish \eqref{eq:EL}, we note that convexity of
  $v\mapsto \Psi(S^i,v)+\<D\Phi^\eps(S^i),v\>$ implies that the subdifferential at the
  minimum must contain $0$, and therefore, by the characterisation of $\partial_v\psi$ given in
  \eqref{eq:Subgradient_psi^*}, it follows that
\begin{equation}\label{eq:IEL}
  \int_{S^i} \nabla_\tau v^i\cdot A(b,\tau)\nabla_\tau w + \b(D_\tau^\perp\psi(b,\tau,v^i)-\PK\b)\cdot w\,\dd\Haus^1=0\qquad\text{for any }w\in H^1(S^i;\R^3).
\end{equation}
Standard properties of convex functions entail that
\begin{equation*}
  \psi(b,\tau,0)\geq \psi(b,\tau,v)-\xi\cdot v \qquad\text{for any }\xi\in\partial\psi(b,\tau,v),
\end{equation*}
so since $\psi(b,\tau,0)=0$ by {\bf (C$_2$)}, it follows that
\begin{equation*}
  D_\tau^\perp\psi(b,\tau,v)\cdot v\geq \psi(b,\tau,v)\geq \tfrac12\beta|v|^2 \qquad\text{for any }v\text{ such that }v\cdot\tau=0,
\end{equation*}
where we have applied {\bf(G)}. Setting $w=v^i$ in \eqref{eq:IEL} and bounding the left--hand side below as in \eqref{eq:Psi*LowerBound}, we thereby obtain
\begin{equation*}
  \tfrac12\gamma\|v^i\|_{H^1}^2 \leq \b\|\PK(S^i)\b\|_{L^2}\|v^i\|_{L^2}\leq \b\|\PK(S^i)\b\|_{L^2}\|v^i\|_{H^1}.
\end{equation*}
The bound \eqref{eq:APVelBound} then follows directly from the estimate established in Corollary~\ref{th:PKL2}.
\end{proof}

\noindent
Considering the details of the proof above, we make two remarks:
\begin{itemize}
\item Estimate \eqref{eq:APVelBound} hinges upon the $L^2$ bound on $\PK$ made in
  Corollary~\ref{th:PKL2}, which in turn relies upon the $L^\infty$ bound \eqref{eq:PKLInfty}
  proved in Lemma~\ref{th:Kernels}. Any improvement of \eqref{eq:PKLInfty} would therefore entail
  an improved bound on $v^i$.
\item The choice to assume quadratic growth of $\psi$ in {\bf(G)} allows us to directly obtain
  an $H^1$ estimate on $v^i$; if a weaker growth condition was assumed, we would need to apply a
  Poincar\'e--type inequality to obtain a similar estimate. Since such an inequality would
  inevitably depend upon $\Mass(S)$, this would render some aspects of the arguments which follow
  more technical.
\end{itemize}

As a consequence of \eqref{eq:APVelBound}, we may apply the Cauchy--Schwarz inequality to show the
following corollary.

\begin{corollary}
  The solution to the minimisation problem \eqref{eq:Timestep} satisfies
\begin{equation}\label{eq:LengthRateBound}
  \|\nabla_\tau v^i\|_1
  \leq \Mass(S^i)^{1/2}\|\nabla_\tau v^i\|_2
  \leq \frac{C\,\Mass(S^i)\MRatio(S^i)\|b\|_{L^\infty}}{ \eps\min(\alpha,\beta)}\log\bg|1+\frac{2\,\Mass(S^i)}{\eps\,\MRatio(S^i)}\bg|.
\end{equation}
\end{corollary}\medskip

\noindent
Estimate \eqref{eq:LengthRateBound} will be important later, as it provides a control on the
maximal growth rate of $\Mass(S^i)$.

\subsection{Properties of approximate solutions}
Now that the existence of $v^i:S^i\to\R^3$ has been established, we wish to define $S^{i+1}$ as the
pushforward of $S^i$ under the mapping $\delta U^i(s):=\id(s)+\delta t^i\,v^i(s)$.
At present, we have only established that $v^i$ is in $H^1(S^i;\R^3)$; this entails that
$\delta U^i$ is continuous as a mapping from $S^i$ to $\R^3$, but in order to be sure that $S^{i+1}\in\Adm$, we must show that $\delta U^i$ is
Lipschitz, and hence we must develop a regularity theory for $v^i$. The following result
establishes several crucial properties of $v^i$.

\begin{lemma}
  \label{th:Regularity}
  If $S^i$ is a finite union of $C^{k,\gamma}$ curves with $k\geq1$ and $0<\gamma\leq 1$, then the solution $v^i$ to the minimisation
  problem \eqref{eq:Timestep} is
  $C^{k,\gamma}$, and moreover we have the bounds
\begin{gather}
  \|v^i\|_{L^\infty} \leq \frac{C\sqrt{1+2\,\Mass(S^i)}\MRatio(S^i)\|b\|_{L^\infty}}
      {\eps\,\min(\alpha,\beta)}\log\bg|1+\frac{2\,\Mass(S^i)}{\eps\,\MRatio(S^i)}\bg|
  \label{eq:vUniformBound}
  \\
  \|\nabla_\tau v^i\|_{L^\infty} \leq
  \frac{C''}{\eps\,\alpha}\|b\|_{L^\infty}
  \bg(1+\frac{\sqrt{1+2\,\Mass(S^i)}}{\min(\alpha,\beta)}\bg)\Mass(S^i)\MRatio(S^i)
  \log\bg|1+\frac{2\,\Mass(S^i)}{\eps\,\MRatio(S^i)}\bg|\label{eq:DvUniformBound} \\
    \begin{aligned}{}
      [\nabla_\tau v]_\gamma &\leq
      \b(\Mass(S)^{1-\gamma}+[\tau]_\gamma\Mass(S)\b)\frac{C'''}{\eps\,\alpha}\|b\|_{L^\infty}\bg(1+\frac{\sqrt{1+2\,\Mass(S)}}
    {\min(\alpha,\beta)}\bg)
    \MRatio(S)\log\bg|1+\frac{2\,\Mass(S)}{\eps\,\MRatio(S)}\bg|.
  \end{aligned}\label{eq:DvCalpha}
    \end{gather}
\end{lemma}\medskip

To prove this result, we pull back to a flat domain, recast the resulting equation
as an ODE system, and then use the properties assumed of $A$ and $\psi$ along with some elementary
integral bounds.

\begin{proof}
  We first prove regularity, then proceed to obtain the stated estimates. Since they are fixed
  throughout this proof, we suppress superscripts, writing $v$ and $S$ in place of $v^i$ and
  $S^i$.\medskip

  \noindent
  \emph{Regularity.}
  Since $S$ is assumed to be a union of $C^{k,\gamma}$ curves, there exists a $C^{k,\gamma}$
  diffeomorphism $g$ which maps
  the interval $(-a,a)$ to a neighbourhood of $s\in S$. Without loss of generality, we may assume
  $g$ is an arc--length parametrisation, so $g' = g^\#\tau$ on $(-a,a)$, and therefore
  $|g'|=1$ since $\tau=1$. Defining $V:(-a,a)\to \R^3$ to be the pullback of $v$ by $g$,
  i.e. $V := g^\# v$, we find it has weak derivative
  \begin{equation*}
    g^\#\nabla_\tau v = V'.
  \end{equation*}
  We recall that the equation satisfied by $v$ on $S$ is
  \begin{equation*}
    -\div_\tau\b[A(b,\tau)\nabla_\tau v\b]
    +D^\perp_\tau\psi(b,\tau,v)=\PK,
  \end{equation*}
  so defining $B = g^\#b$ and $F=g^\#\PK$ and `pulling back' the equation, we find that
  $V:(-a,a)\to\R^3$ must satisfy 
  \begin{equation}\label{eq:PulledBack}
    -\b[A\b(B(r),g'(r)\b)V'(r)\b]'+D_\tau^\perp\psi\b(B(r),g'(r),V(r)\b)
    = F(r)
\end{equation}
almost--everywhere on $(-a,a)$. As remarked at the end of \S\ref{sec:peach-koehler-force}, the
fact that $S$ is assumed to by $C^{k\,gamma}$ implies that $F\in C^{k-1,\gamma}\b((-a,a);\R^3\b)$.

We define the auxiliary function $\sigma:=g^\#\b(A\b(b,\tau\b)\nabla_\tau v\b)$, so that
$V' = A(B,g')^{-1}\sigma$.
Solving \eqref{eq:PulledBack} for $V$ entails that $V$ and $\sigma$ must satisfy the system of
equations
\begin{equation}\label{eq:PulledBackSystem}
  \begin{aligned}
    V'(r) &=A\b(B(r),g'(r)\b)^{-1}\sigma(r)\\
    \sigma'(r)&= D_\xi\psi\b(B(r),g'(r),V(r)\b)-F(r).
  \end{aligned}
\end{equation}

  
  Now, since $v\in H^1(S;\R^3)\subset C^{0,\frac12}(S;\R^3)$, and
  $g\in C^{1,\gamma}\b((-a,a);S\b)\subset C^{0,1}\b((-a,a);S\b)$, it follows that
  $V=v^\#g\in C^{0,\frac12}\b((-a,a);\R^3\b)$. The regularity assumptions on $\psi$ and
  the second equation therefore entail that
  $\sigma\in C^{1,\eta}\b((-a,a);\R^3\b)$, where $\eta=\min\{\tfrac12,\gamma\}$, and the regularity
  assumptions on $A$ applied to the first equation
  entail that $V\in C^{1,\gamma}\b((-a,a);\R^3\b)$. Bootstrapping, we ultimately find that
  $V,\sigma\in C^{k,\gamma}\b((-a,a);\R^3\b)$.

  This local argument entails $v\in C^{k,\gamma}(S;\R^3)$ via a finite covering of $S$, which
  is possible since $S\in\Adm$ is a finite union of $C^{k,\gamma}$ curves.\medskip

  \noindent
  \emph{Uniform bound.}
  Now, taking the inner product between $V$ and $V'$
  and then integrating and applying the Cauchy--Schwarz inequality, we obtain
  \begin{equation*}
    \tfrac12|v(s_1)|^2-\tfrac12|v(s_0)|^2 = \int_{g(s_0)}^{g(s_1)} V(r)\cdot V'(r)\,\dd r
    \leq  \|V\|_{L^2}\|V'\|_{L^2} = \|v\|_{L^2}\b\|\nabla_\tau v\b\|_{L^2}.
  \end{equation*}
  Integrating with respect to $s_0$, dividing by $\Mass(S)$, and using \eqref{eq:APVelBound}, we
  find that
  \begin{equation}\label{eq:vUniformBoundProof}
    \begin{aligned}
      |v(s_1)|^2
      &\leq 2 \|v\|_{L^2}\|\nabla_\tau v\|_{L^2}+ \frac{\|v\|_{L^2}^2}{\Mass(S)},\\
      &\leq \frac{C^2\b(1+2\,\Mass(S)\b)\MRatio(S)^2\|b\|^2_{L^\infty}}
      {\eps^2\min(\alpha,\beta)^2}\log\bg|1+\frac{2\,\Mass(S)}{\eps\,\MRatio(S)}\bg|^2.
  \end{aligned}
\end{equation}
  which leads directly to \eqref{eq:vUniformBound}.
  \medskip

  \noindent
  \emph{Uniform gradient bound.}
  If $V$ and $\sigma$ solve \eqref{eq:PulledBackSystem} on $(-a,a)$, then by integrating the
  second equation, we find that
  \begin{equation}\label{eq:sigHolder}
    |\sigma(r_1)-\sigma(r_0)| \leq C' \b(\|F\|_{L^\infty} + \|V\|_{L^\infty}\b) |r_1-r_0|
  \end{equation}
  so $\sigma\in C^{0,1}\b((-a,a);\R^3)$.
  Moreover, we may use \eqref{eq:vUniformBound} and the definition of $F$ as a pullback of
  $\PK$ to find
  \begin{equation}\label{eq:sigUniformBound}
    \begin{aligned}
      \|\sigma\|_{L^\infty}
      &\leq C'\B(\|F\|_{L^\infty} + \|v\|_{L^\infty}\B)\smfrac12 \Mass(S)\\
      &\leq \frac{C''}{\eps}\|b\|_{L^\infty}\bg(1+\frac{\sqrt{1+2\,\Mass(S)}}
      {\min(\alpha,\beta)}\bg)
      \Mass(S)\MRatio(S)\log\bg|1+\frac{2\,\Mass(S)}{\eps\,\MRatio(S)}\bg|
    \end{aligned}
  \end{equation}
  Using the definition of $\sigma$, and noting that {\bf(C$_1$)} implies that
  $\|A^{-1}\|_{L^\infty}\leq \alpha^{-1}$, we therefore obtain
  \begin{equation*}
    \|\nabla_\tau v\|_{L^\infty}\leq
    \frac{C''}{\eps\,\alpha}\|b\|_{L^\infty}\bg(1+\frac{\sqrt{1+2\,\Mass(S)}}
      {\min(\alpha,\beta)}\bg)
      \Mass(S)\MRatio(S)\log\bg|1+\frac{2\,\Mass(S)}{\eps\,\MRatio(S)}\bg|,
  \end{equation*}
  which is \eqref{eq:DvUniformBound}.
  \medskip
  
  \noindent
  \emph{Uniform bound on H\"older seminorm.}
  Applying the assumption {\bf(R)} to deduce that
  $\b|A^{-1}(b,\tau_1)-A^{-1}(b,\tau_2)\b|\leq L|\tau_1-\tau_2|$ for some $L$, we have
  \begin{align*}
    |V'(r_1)-V'(r_0)|
    &= \b|A\b(B(r_1),g'(r_1)\b)^{-1}\sigma(r_1)-A\b(B(r_0),g'(r_0)\b)^{-1}\sigma(r_0)\b|\\
    &\leq \b|A\b(B(r_1),g'(r_1)\b)^{-1}\sigma(r_1)-A\b(B(r_1),g'(r_1)\b)^{-1}\sigma(r_0)\b|\\
    &\qquad\qquad+\b|A\b(B(r_1),g'(r_1)\b)^{-1}\sigma(r_0)-A\b(B(r_1),g'(r_0)\b)^{-1}\sigma(r_0)\b|\\
    &\leq\alpha^{-1}|\sigma(r_1)-\sigma(r_0)|+L\,\b|g'(r_1)-g'(r_0)\b|\|\sigma\|_{L^\infty}\\
    &\leq \B(\alpha^{-1}[\sigma]_\gamma+L\,[g']_\gamma\|\sigma\|_{L^\infty}\B)|r_1-r_0|^\gamma
  \end{align*}
  Using \eqref{eq:sigHolder}, we find that
  \begin{equation*}
    [\sigma]_\gamma \leq \frac{C''}{\eps}\|b\|_{L^\infty}\bg(1+\frac{\sqrt{1+2\,\Mass(S)}}
    {\min(\alpha,\beta)}\bg)
    \Mass(S)^{1-\gamma}\MRatio(S)\log\bg|1+\frac{2\,\Mass(S)}{\eps\,\MRatio(S)}\bg|,
  \end{equation*}
  and estimating the other term using \eqref{eq:sigUniformBound}, we obtain \eqref{eq:DvCalpha}.
\end{proof}

Lemma~\ref{th:Regularity} guarantees the spatial regularity of any approximate solution $U(t)$
defined via the procedure prescribed in \eqref{eq:Timestep}, and therefore ensures that such
approximate solutions are well--defined. Our next step will be to prove that approximate solutions
converge as $\max_i\{\delta t^i\}\to0$, and that the limit satisfies \eqref{eq:DDD}.

\subsection{Convergence of approximate solutions}
Our approach to proving convergence is via compactness; this requires us to prove appropriate
uniform \emph{a priori} bounds on approximate solutions, which will subsequently allow us to employ
the Arzel\`a--Ascoli theorem. We remark that all bounds on approximate solutions derived thus far
depend upon $\Mass(S)$ and $\MRatio(S)$, and therefore it is these quantities we must bound; the
following lemma therefore establishes a bound on the growth of $\Mass(U(t)_\#S)$.

\begin{lemma}
  \label{th:MassMRatioGrowthBounds}
  If $S^0\in\Adm$ with $\MRatio(S^0)<+\infty$, then for any $\rho>\MRatio(S^0)$ and $M>\Mass(S^0)$,
  and all $\delta>0$ sufficiently small, there exists $T(M,\rho,\delta)>0$ such that any
  approximate solution of DDD
  (in the sense described in \S\ref{th:vExistence}) with $\max_i\{\delta t^i\}\leq \delta$
  satisfies
  \begin{equation*}
    \Mass(U(t)_\#S^0)\leq M\quad\text{and}\quad\MRatio(U(t)_\#S^0)\leq \rho
    \quad\text{for all }t\in[0,T].
  \end{equation*}
\end{lemma}

\begin{proof}
  The proof is divided into first obtaining a uniform bound on the mass growth, then using this
  bound to guarantee a bound on the growth of the mass ratio.\medskip

  \noindent
  \emph{Uniform mass bound.}
  Our first step is to establish a uniform bound on the mass.
  By definition, $U(t)_\#S^0 = (\id+(t-t_i)v^i)_\#S^i$ for $t\in(t_i,t_{i+1})$, and so
  \eqref{eq:LengthRateBound} implies
  \begin{align*}
    \Mass(U(t)_\#S^0) &= \Mass\b((\id+(t-t^i)v^i)_\#S^i\b)\\
    &=\|\tau^i+(t-t^i)\nabla_\tau v^i\|_{L^1(S^i)}\\
    &\leq \Mass(S^i)+(t-t^i) \frac{C\MRatio(S^i)\,\Mass(S^i)\|b\|_{L^\infty}}
      { \eps\min(\alpha,\beta)}\log\bg|1+\frac{2\,\Mass(S^i)}{\eps\,\MRatio(S^i)}\bg|.
  \end{align*}
  Estimating $\log|1+x|\leq x$ for $x\geq 0$, and employing the comparison principle for ODEs,
  we find that $\Mass(U(t)_\# S^0)$ must be bounded above by the solution to 
  \begin{equation*}
    m'(t) = \frac{2\,C\,
      \|b\|_{L^\infty}}{\min(\alpha,\beta)}\frac{m(t)^2}{\eps^2}
    \qquad m(0)=\Mass(S^0).
  \end{equation*}
  Since this is a separable ODE, we may check that 
  \begin{equation}\label{eq:MassBound}
    m(t)=\bg(\frac{1}{\Mass(S^0)}-\frac{2\,C\,t\,
      \|b\|_{L^\infty}}{\eps^2\min(\alpha,\beta)}\bg)^{-1}
  \end{equation}
  for all $t\in[0,T]$, and therefore it follows that 
  \begin{equation*}
    \Mass(U(t)_\#S^0)\leq M\quad\text{for all}\quad t\leq T(M)=C_1\frac{\b(M-\Mass(S^0)\b)}
    {M\,\Mass(S^0)},
  \end{equation*}
  where $C_1$ is a constant depending on $\|b\|_{L^\infty}$, $\eps$ and $\min(\alpha,\beta)$.
  \medskip

  \noindent
  \emph{Uniform mass ratio bound.}
  Suppose that $U(t)$ is an approximate solution on $[0,T(M)]$, with
  $\max_i\{\delta t^i\}\leq \delta\leq T(M)$. The bounds
  \eqref{eq:vUniformBound} and \eqref{eq:DvUniformBound} entail that for $t$ small enough,
  the map $\delta U^i(t):[t^i,t^{i+1}]\times S^i\to\R^3$ defined by $\delta U^i(t):=\id+(t-t^i)v^i$ is
  satisfies
  \begin{equation*}
    \frac{|\delta U^i(t,s_1)-\delta U^i(t,s_0)|}{|s_1-s_0|}
    \leq \bg|1+(t-t^i)\,C_2\MRatio(S^i)\log\B|1+\frac{2\,M}{\eps\,\MRatio(S^i)}\B|\bg|
  \end{equation*}
  for any $s_0,s_1\in S^i$ for some constant $C_2>0$. Moreover, since $\MRatio(S^i)\geq 1$ for any
  $S^i\neq\emptyset$, as observed in \eqref{eq:MRatioLowerBound}, we may bound the logarithmic
  terms by $\log|1+2M/\eps|$, finding that
  \begin{equation}\label{eq:LipschitzBound}
    \frac{|\delta U_i^t(s_1)-\delta U_i^t(s_0)|}{|s_1-s_0|}
    \leq \b|1+C_3\,(t-t^i)\MRatio(S^i)\b|.
  \end{equation}
  for some $C_3>0$.

  Next, recalling an argument made in \S3.2.17 of \cite{Federer}, supposing that
  $F:S\to\R^3$ satisfies
  \begin{equation*}
    \frac{|F(s_1)-F(s_0)|}{|s_1-s_0|} \leq \lambda \qquad\text{for all }s_0,s_1\in\overline{B_r(s)},
  \end{equation*}
  for some $\lambda>1$, it follows that
  \begin{equation*}
    F_\#S\cap \overline{B_r\b(F(s)\b)}\subset F_\#\b(S\cap \overline{B_{\lambda r}(s)}\b),
  \end{equation*}
  and therefore
  \begin{equation*}
    \frac{\Haus^1\B(F_\#S\cap \overline{B_{r}\b(F(s)\b)}\B)}{r}
    \leq \lambda^2\frac{\Haus^1\b(S\cap\overline{B_{\lambda r}(s)}\b)}{\lambda\,r}.
  \end{equation*}
  Noting the connection between the Hausdorff measure of the support of an integral current and
  its mass made in \eqref{eq:MassvsMeasure}, we may take suprema over $r\geq0$ and $s\in\R^3$
  to obtain
  \begin{equation}\label{eq:MassRatioGrowth}
    \MRatio(F_\#S)
    \leq \lambda^2 \MRatio(S).
  \end{equation}
  In particular, $\delta U^i(t)$ satisfies the conditions on $F$
  with
  \begin{equation*}
    \lambda=1+ C_3\,(t-t^i)\,\MRatio(S^i),
  \end{equation*}
  and so we find that
  \begin{equation*}
    \frac{\MRatio\b(\delta U^i(t)_\#S^i\b)-\MRatio(S^i)}{t-t^i}
    \leq C_3\b(2+C_3\,\delta\,\MRatio(S^i)\b) \MRatio(S^i)^2.
  \end{equation*}
  Employing the comparison principle for ODEs once more, and using the fact that, by construction,
  $\MRatio(U(t)_\#S^0)=\MRatio(\delta U^i(t)_\#S^i)$, we find that $\MRatio(U(t)_\#S^0)$ is bounded above
  by the solution to
  \begin{equation*}
    \rho'(t) \leq C_3\b(2+C_3\delta \rho(t)\b)\rho(t)^2,\qquad\text{with}\quad\rho(0)=\MRatio(S^0).
  \end{equation*}
  Again, being a separable ODE, we may integrate to find
  \begin{equation*}
    \frac{1}{2C_3 \MRatio(S^0)} - \frac\delta 4 \log\bg|C_3\delta + \frac{2}{\MRatio(S^0)}\bg|-\frac{1}{2C_3 \rho(t)} + \frac\delta 4 \log\bg|C_3\delta + \frac{2}{\rho(t)}\bg| = t,
  \end{equation*}
  and thereby we see that for any
  \begin{equation*}
    t\leq T(\rho,M,\delta) = \frac{\rho-\MRatio(S^0)}{2C_3(M)\MRatio(S^0)}
    +\frac\delta 4 \log\bg|\frac{C_3(M)\delta + 2/\rho}{C_3(M)\delta+2/\MRatio(S^0)}\bg|,
  \end{equation*}
  the result holds.
\end{proof}

The result of Lemma~\ref{th:MassMRatioGrowthBounds} entails that, given an upper limit on the
mass and mass ratio, and a maximum step size, there exists an infinite family of approximate
solutions on some time interval $[0,T]$ with $T>0$. This is a crucial result which allows us to
establish existence by compactness in the following lemma.

\begin{lemma}
  \label{th:Existence}
  Suppose $S^0\in\Adm$ is $C^{1,\gamma}$ for some $\gamma\in(0,1]$ and fix $\rho$ and $M$ such that
  $\Mass(S^0)<M<+\infty$ and $\MRatio(S^0)<\rho<+\infty$; then there exists $T>0$ such that
  there exists a unique $C^0\b([0,T];C^1(S^0;\R^3)\b)$ solution to \eqref{eq:DDD}, and
  moreover $\Mass(U(t)_\#S^0)\leq M$ and $\MRatio(U(t)_\#S^0)\leq \rho$ for all $t\in[0,T]$.
\end{lemma}

\begin{proof}
  In Lemma~\ref{th:MassMRatioGrowthBounds}, we established the existence of $T$ such that if
  $U_n:[0,T]\times\R^3$ is an approximation solution of DDD with $\max_i\{\delta t^i_n\}\leq\delta$,
  then $\Mass\b(U_n(t)_\#S^0\b)\leq M$ and $\MRatio\b(U_n(t)_\#S^0\b)\leq \rho$. Therefore, taking
  a sequence of such solutions with $\max_i\{\delta t^i_n\}\to 0$, we wish to show that this
  sequence contains a convergent subsequence via an application of the Arzel\`a--Ascoli
  Theorem (or equivalently, the compactness properties of H\"older spaces).\medskip

  \noindent
  \emph{Compactness.}
  When combined with the result of Lemma~\ref{th:MassMRatioGrowthBounds}, the bounds established
  in Lemma~\ref{th:Regularity} entail that
  \begin{equation*}
    \b\|\dot{U}_n(t)\b\|_{C^{1,\gamma}}\leq C(M,\rho,\delta)\quad\text{for all }n\in\N;
  \end{equation*}
  in turn, this entails that
  \begin{equation}\label{eq:LipschitzC1gamma}
    \b\|U^{t_1}_n-U^{t_0}_n\b\|_{C^{1,\gamma}}\leq \int_{t_0}^{t_1}\b\|\dot{U}_n(t)\b\|_{C^{1,\gamma}}\dd t
    \leq C(M,\rho,\delta)\,|t_1-t_0|.
  \end{equation}
  Setting $t_0=0$, it follows that $U_n:[0,T]\times S^0\to\R^3$ are uniformly bounded in
  $C^{0,1}([0,T];C^{1,\gamma}(S^0;\R^3))$, from which it follows that the sequence of approximate
  solutions is compact in $C^0\b([0,T];C^1(S^0;\R^3)\b)$. We may therefore extract a convergent
  subsequence, which we do not relabel; we denote the limit $U_\infty$.

  Moreover, the same results also imply that
  $U_n(t)^\#v^t_n$ is uniformly bounded in $L^2\b([0,T];H^1(S^0;\R^3)\b)$. It follows that we may
  extract a further subsequence (which again, we do not label) such that $U_n(t)^\#v^t_n$ converges
  weakly in $H^1(S^0;\R^3)$ for almost every $t\in[0,T]$.

  Now that we have demonstrated the existence of a candidate limit, we must demonstrate that it
  solves \eqref{eq:DDD}
  \medskip

  \noindent
  \emph{Convergence of dissipation potential.}
  By virtue of the fact that $U_n\to U_\infty$ in $C^{0,1}\b([0,T];C^1(S^0)\b)$,
  \begin{equation*}
    U_n(t)^\#\tau_n(t) = \nabla_\tau U_n(t)\to \nabla_\tau U_\infty(t)=U_\infty(t)^\#\tau_\infty(t),\quad\text{uniformly for }t\in[0,T]\text{ as }n\to\infty,
  \end{equation*}
  where $\tau_n(t):U_n(t)_\# S^0\to\Sph^2$ and $\tau_\infty(t):U_\infty(t)_\#S^0\to\Sph^2$ are tangent
  fields.
  Recalling the definition made in \eqref{eq:Timestep}, by construction,
  approximate solutions satisfy
  \begin{equation*}
    \partial_v\Psi\b(U_n(t^i_n)_\#S^0,U_n(t^i_n)^\#\dot{U}_n(t^i_n)\b)
    =\PK(U_n(t^i_n)_\#S^0)
  \end{equation*}
  for each $t^i_n$.

  Applying assumption {\bf(R)}, we have
  \begin{equation*}
    A\b(U_n(t)^\# b,U_n(t)^\#\tau_n\b)\to A\b(U_\infty(t)^\# b,U_\infty(t)^\# \tau_\infty\b)
    \quad\text{in }L^\infty(S^0;\R^{3\times 3})
  \end{equation*}
  uniformly in $t$ as $n\to\infty$, and therefore if $V_n$ is a sequence of functions in
  $H^1(S;\R^3)$ such that $\nabla_\tau V_n\to \nabla_\tau V_\infty$ weakly
  in $L^2(S^0;\R^3)$, we have
  \begin{multline}\label{eq:LSCATerm}
    \liminf_{n\to\infty}\int_{S^0}\tfrac12 A\b(U_n(t)^\#b,U_n(t)^\#\tau_n(t)\b):
    [\nabla_\tau V_n,\nabla_\tau V_n]\,\dd\Haus^1\\
    \geq \int_{S^0}\tfrac12 A\b(U_\infty(t)^\#b,U_\infty(t)^\#\tau_\infty(t)\b):
    [\nabla_\tau V_\infty,\nabla_\tau V_\infty]\,\dd\Haus^1.
  \end{multline}
  Furthermore, the convexity and regularity assumptions {\bf(C$_2$)} and {\bf(R)} imply that if
  $V_n\in L^2(S^0;\R^3)$ is a sequence of functions such that $V_n\to V_\infty$ weakly in
  $L^2(S^0;\R^3)$ as  $n\to\infty$, then
  \begin{equation}\label{eq:varphiConvergence}
    \liminf_{n\to\infty}\int_{S^0}\psi\b(U_n(t)^\#b,U_n(t)^\#\tau_n(t),V_n\b)\dd\Haus^1
    \geq \int_{S^0}\psi\b(U_\infty(t)^\#b,U_\infty(t)^\#\tau_\infty(t),V_\infty\b)\dd\Haus^1.
  \end{equation}
  Together, \eqref{eq:LSCATerm}, \eqref{eq:varphiConvergence} and the fact that
  $U_n(t)^\# v^t_n$ converges weakly in $L^2([0,T];H^1(S^0;\R^3))$ imply that
  \begin{equation}\label{eq:EntropyConvergence}
    \liminf_{n\to\infty} U_n(t)^\#\Psi\b(U_n(t)_\# S^0;v_n\b) \geq  U_\infty(t)^\#\Psi\b(U_\infty(t)_\# S^0; v_\infty\b)
  \end{equation}
  for almost every $t\in[0,T]$.\medskip
  
  \noindent
  \emph{Convergence of Peach--Koehler force.}
  Applying Lemma~\ref{th:PKContinuity} and the fact that $U_n(t)\to U_\infty(t)$ in $C^{1}(S^0;\R^3)$
  uniformly in $t$, we find that
  \begin{equation}\label{eq:PKConvergence}
    U_n(t)^\#\PK\b(U_n(t)_\#S^0\b)\to U_\infty(t)^\#\PK\b(U_\infty(t)_\#S^0\b)
    \quad\text{in }L^\infty(S^0;\R^3)
  \end{equation}
  as $n\to\infty$ uniformly in $t$. Since $U_n(t)^\# v_n^t\to U_\infty(t)^\# v_\infty^t$ weakly in
  $L^2\b([0,T];H^1(S^0;\R^3)\b)$, we further obtain that
  \begin{equation}\label{eq:EnergyDecreaseConvergence}
    \b(U_n(t)^\#\PK\b(U_n(t)_\#S^0\b),U_n(t)^\# v^t_n\b)_{L^2}\to
    \b(U_\infty(t)^\#\PK\b(U_\infty(t)_\#S^0\b),U_\infty(t)^\# v^t_\infty\b)_{L^2}
  \end{equation}
  as $n\to\infty$ for almost every $t\in[0,T]$.
  \medskip

  \noindent
  \emph{Convergence of dissipation potential.}
  Since $\Psi(S,\cdot)$ is a strictly convex functional defined on $H^1(S;\R^3)$, it follows
  that it has a convex conjugate $\Psi^*(S,\cdot):H^1(S;\R^3)^*\to\R$, defined to be
  \begin{equation*}
    \Psi^*(S,\xi) = \sup\b\{\<\xi,v\>-\Psi(S,v):v\in H^1(S;\R^3)\b\}.
  \end{equation*}
  Employing this definition, we obtain
  \begin{equation*}
    \begin{aligned}
      \lim_{n\to\infty} U_n(t)^\#\Psi^*(U_n(t)_\#S^0,\xi) &= \lim_{n\to\infty}\,
      \sup\b\{\<\xi,v\>-U_n(t)^\#\Psi(U_n(t)_\#S^0,v):v\in H^1(S;\R^3)\b\}\\
      &=\sup\{\<\xi,v\>-\lim_{n\to\infty}U_n(t)^\#\Psi(U_n(t)_\#S^0,v):v\in H^1(S;\R^3)\b\}\\
      &= \sup\b\{\<\xi,v\>-U_\infty(t)^\#\Psi(U_\infty(t)_\#S^0,v):v\in H^1(S;\R^3)\b\}\\
      &=U_\infty(t)^\#\Psi^*(U_\infty(t)_\#S^0,\xi).
    \end{aligned}
  \end{equation*}
  Moreover, since \eqref{eq:PKConvergence} holds, it follows that
  \begin{equation*}
    U_n(t)^\#D\Psi^\eps(U_n(t)_\#S^0)
  \to U_\infty(t)^\#D\Psi^\eps(U_\infty(t)_\#S^0)\quad\text{in }H^1(S^0;\R^3)^*,
  \end{equation*}
  and we conclude that
  \begin{equation}\label{eq:DissConvergence}
    U_n(t)^\#\Psi\B(U_n(t)_\#S^0,U_n(t)^\#\PK(U_n(t)_\#S^0)\B)\to
    U_\infty(t)^\#\Psi\B(U_\infty(t)_\#S,U_\infty(t)^\#\PK\b(U_\infty(t)_\#S^0\b)\B)
  \end{equation}
  as $n\to\infty$.\medskip

  \noindent
  \emph{Conclusion.}
  Finally, by using standard properties of the Legendre--Fenchel transform \cite{Rockafellar,AGS,M06},
  we note that \eqref{eq:DDD} is equivalent to requiring that
  \begin{equation*}
    U(t)^\#\Psi\b(U(t)_\# S^0, v^t\b)+U(t)^\#\Psi^*\B(U(t)_\# S^0,-D\Psi^\eps\b(U(t)_\# S^0\b)\B)-\B\< D\Psi^\eps\b(U(t)_\#S^0),v^t\B\>=0
  \end{equation*}
  for almost every $t\in[0,T]$.
  By considering this expression with $U_n(t)$ in place of $U(t)$, and $v^t_n$
  in place of $v^t$, we may combine \eqref{eq:EntropyConvergence},
  \eqref{eq:EnergyDecreaseConvergence} and \eqref{eq:DissConvergence}, to pass to the
  limit, demonstrating that
  \begin{multline*}
    0 = \int_0^T \bg[U_\infty(t)^\#\Psi\b(U_\infty(t)_\# S^0;v^t_\infty\b)
    +U_\infty(t)^\#\Psi\b(U_\infty(t)_\#S^0,U_\infty(t)^\#\PK(U_\infty(t)_\#S^0)\b)\\
    +\B(U_\infty(t)^\#\PK\b(U_\infty(t)_\#S^0\b),U_\infty(t)^\# v^t_\infty\B)_{L^2}\bg]\dd t.
  \end{multline*}
  This entails that, for almost every $t$, we have
  \begin{equation*}
    \partial_v\Psi\b(U_\infty(t)_\# S^0;v^t_\infty\b)\ni\PK\b(U_\infty(t)_\#S^0\b),
  \end{equation*}
  and since $U_\infty(t)^\#v^t_\infty = \lim_{n\to\infty} U_n(t)^\#v^t_n = \lim_{n\to\infty}\dot{U}_n(t)
  = \dot{U}_\infty(t)$, we have proved that $U_\infty$ solves \eqref{eq:DDD}.

  To demonstrate uniqueness of the limit, we note that assumptions {\bf(C$_1$)} and {\bf(C$_2$)}
  entail that the functional $\Psi(S,v)$ is strictly convex on $H^1(S;\R^3)$, and therefore a
  standard argument guarantees that the above limit procedure is independent of the subsequence
  chosen.
\end{proof}

\subsection{Conclusion of the proof}
Now that we have proved existence for a finite time, we show that we may extend the solution to a
possibly infinite time by demonstrating that the total mass of the dislocation configuration is
bounded as long as the mass ratio remains bounded.

\begin{lemma}
  \label{th:MassGrowthBound}
  If $U:[0,T]\times S^0\to\R^3$ is a solution of DDD in the sense described in
  \eqref{eq:DDD}, such that $\MRatio(U(t)_\#S^0)\leq \rho$ for all $t\in[0,T]$, then
  we have the uniform bound
  \begin{gather*}
    \Mass(U(t)_\#S^0)\leq \min\bg\{\frac{\Mass(S^0)}{1-C_1(b,\eps)\,\Mass(S^0)\,T},\frac{\eps}{2} \bg|1 + \frac{2\,\Mass(S^0)}{\eps}\bg|^{\exp(C_2(b,\eps,\rho)T)}\bg\},\\
    \quad\text{where}
    \quad C_1(b,\eps) = \frac{2\,C\,
      \|b\|_{L^\infty}}{\eps^2\min(\alpha,\beta)}\quad\text{and}\quad
    C_2(b,\eps,\rho) = \frac{C\,\rho\,
      \|b\|_{L^\infty}}{\eps\min(\alpha,\beta)}.
  \end{gather*}
\end{lemma}
\begin{proof}
  Using the fact that $v^t$ must satisfy \eqref{eq:LengthRateBound}, and by assumption,
  $1\leq \MRatio(U(t)_\#S^0)\leq\rho$, we have
  \begin{align*}
    \frac{\dd}{\dd t}\Mass(U(t)_\#S^0) \leq \frac{C\rho\,\Mass(U(t)_\#S^0)\|b\|_{L^\infty}}
      { \eps\min(\alpha,\beta)}\log\bg|1+\frac{2\,\Mass(U(t)_\#S^0)}{\eps}\bg|.
  \end{align*}
  Employing the comparison principle for
  ODEs, it follows that for all $t\in[0,T]$, $\Mass(U(t)_\# S^0)$ must be bounded above by the
  solution to
  \begin{equation*}
    m'(t)=\frac{C\,\rho\,
      \|b\|_{L^\infty}}{2\min(\alpha,\beta)}\bg(1+\frac{2m(t)}{\eps}\bg)\log\bg|1+\frac{2\,m(t)}{\eps}\bg|,
    \qquad m(0)=\Mass(S^0).
  \end{equation*}
  Integrating, we find that
  \begin{gather*}
    \log\bg|1 + \frac{2\,m(t)}{\eps}\bg| = \log\bg|1 + \frac{2\,\Mass(S^0)}{\eps}\bg|
    \exp\bg(\frac{C\,t\,\rho\,\|b\|_{L^\infty}}{\eps\min(\alpha,\beta)}\bg)
  \end{gather*}
  for all $t\in[0,T]$, which, when combined with \eqref{eq:MassBound} which was used to bound the
  mass in the proof of Lemma~\ref{th:MassMRatioGrowthBounds} now immediately yields the bound
  stated.
\end{proof}

With this result in place, we now deduce the following result, which complete the proof of Theorem~\ref{th:WellPosedness}.

\begin{corollary}
  If $S^0\in\Adm$ satisfies $\MRatio(S^0)<+\infty$, then there exists a unique solution
  $(U,\{v^t\})$ of \eqref{eq:DDD} up until $T:=\sup\b\{t\in\R:\MRatio(U(s)_\#S^0)<+\infty\text{ for all }s\leq t\b\}$.
\end{corollary}

\begin{proof}
  For any $\rho>\MRatio(S^0)$, and $M>\Mass(S^0)$ Lemma~\ref{th:Existence} establishes the existence
  of a unique solution up until the first time $T(M,\rho)$ at which either $\Mass(U(T)_\#S^0)=M$ or
  $\MRatio(U(T)_\#S^0)=\rho$, and we note that this existence time is bounded below by a function
  which is monotone in both $M$ and $\rho$. Lemma~\ref{th:MassGrowthBound} entails that in fact
  $\Mass(U(t)_\# S^0)$ is finite for any $t\leq T^*(\rho)$, where $T^*$ is the first time at which
  $\MRatio(U(T^*)_\#S^0)=\rho$, and thus the maximal existence time is independent of the mass
  constraint. Letting $\rho\to\infty$, we obtain the result.
\end{proof}

\section*{Acknowledgements}
\noindent
\emph{Thanks.}
Thanks go to numerous people who I have discussed this project with over the course of working on
it, including Julian Braun, Maciej Buze, Lucia de Luca, Steve Fitzgerald (who first highlighted \cite{BBS80} to me), Adriana Garroni, James Kermode, Christoph Ortner, Mark Peletier, Ed Tarleton, Florian Theil, and Patrick van Meurs.\medskip

\noindent
\emph{Funding.}
This work was supported by Early Career Fellowship entitled `A Mathematical Study of Discrete Dislocation Dynamics' (ECF-2016-526), awarded by the Leverhulme Trust.

\appendix

\section{Vectors, forms and currents}
\label{appendix}
This appendix recalls various definitions from the theory of currents, as described in
Chapters 1 and 4 of both \cite{Federer} and \cite{Morgan}.

\subsection{Vectors and covectors}
\label{sec:vectors-covectors}
Recall that the usual \emph{exterior product} $\wedge$ is multilinear and alternating, i.e.
it satisfies
\begin{equation*}
  (u+\lambda v)\wedge w = u\wedge w +\lambda(v\wedge w)\quad\text{and}\quad u\wedge v = -v\wedge u.
\end{equation*}
Suppose that $e_1,\ldots,e_n$ form an orthonormal basis of $\R^n$.
The space of $m$--vectors $\Lambda_m\R^n$ is then the span
\begin{equation*}
  \Lambda_m\R^n:=\mathrm{span}\b\{ e_{i_1}\wedge\dots \wedge e_{i_m}\bsep i_j\in\{1,\ldots,n\}, i_1<\ldots < i_m \b\};
\end{equation*}
$m$--vectors should be thought of as describing oriented $m$--dimensional subspaces of $\R^n$.
The space of $m$--covectors, denoted $\Lambda^m\R^n$, is the space of linear functions on
$\Lambda_m\R^n$, i.e. $\Lambda^m\R^n:=[\Lambda_m\R^n]^*$. $\Lambda^m\R^n$ may be identified with
$\Lambda_m(\R^n)^*$, the span of $m$--fold wedge products of dual vectors $e_i^*$, defined to
satisfy $\<e_i^*,e_j\>=\Id_{ij}$.


We may define an inner product,
$(\cdot,\cdot)$ and corresponding norm, $|u| = (u,u)^{1/2}$ on
$\Lambda_m\R^n$, which makes $e_{i_1}\wedge\dots\wedge e_{i_m}$ with
$i_1<\dots<i_m$ an orthonormal basis for the space (see \S1.7 of \cite{Federer}). A similar inner product
can be constructed on $\Lambda^m\R^n$, which makes $e_{i_1}^*\wedge\dots\wedge e_{i_m}^*$ with
$i_1<\dots<i_m$ an orthonormal basis for this space.

While the above definitions are general, throughout this work we will exclusively consider
$n=3$, and $m=1$ or $m=2$, since these are the cases of
interest for the modelling of dislocations. In this case, we have the isometric
isomorphisms
\begin{equation*}
  \R^3\cong\Lambda_2\R^3\cong \Lambda_1\R^3\cong \Lambda^2\R^3\cong\Lambda^1\R^3.
\end{equation*}
In particular, it should be noted that the identification of $\Lambda_2\R^3$ with $\R^3$
corresponds to identifying $u\wedge v$ with the usual vector cross product on $\R^3$.

\subsection{Forms}
\label{sec:forms}
An \emph{$m$--form} is a function $\phi:\R^n\to\Lambda^m\R^n$. Using the basis of $\Lambda^m\R^n$ discussed in \S\ref{sec:vectors-covectors},
any such function may be expressed as
\begin{equation*}
  \phi(x) = \sum_{i_1<\dots<i_m} f_{i_1\ldots i_m}(x)\,e^*_{i_1}\wedge\ldots\wedge e^*_{i_m}.
\end{equation*}
The \emph{exterior derivative} $d\phi$ is the $(m+1)$--form defined via
\begin{equation*}
  d\phi(x) = \sum_{i_1<\dots<i_m} \bg(\sum_{i_{m+1}}f_{i_1\dots i_m,i_{m+1}} e^*_{i_{m+1}}\bg)\wedge(e^*_1\wedge\dots\wedge e^*_m).
\end{equation*}
For any open set $U\subseteq\R^n$, we define the vector space of smooth $m$--forms which are
compactly--supported in $U$ to be
\begin{equation*}
  \mathscr{D}^m(U):=\b\{\phi:U\to\Lambda^m\R^n\bsep \phi\text{ is }C^\infty\text{ with compact support in }U\b\}.
\end{equation*}

\subsection{Currents}
\label{sec:currents}
A current is a generalisation of the notion of a distribution \cite{Freidlander};
whereas distributions act on scalar--valued functions, currents instead
act on the spaces $\mathscr{D}^m(\R^n)$. In particular, an \emph{$m$--dimensional current} (or $m$--current) $T$ is a linear functional which acts on
$\mathscr{D}^m(\R^n)$, and we denote the action of a current on $\phi\in\mathscr{D}^m(\R^n)$ to be
$\<T,\phi\>\in\R$.
The \emph{boundary} of an $m$--dimensional current is the $(m-1)$--current $\partial T$, defined
to be
\begin{equation*}
  \<\partial T,\phi\> = \< T, d\phi\>\quad\text{for all }\phi\in\mathscr{D}^m(\R^n).
\end{equation*}
The support of a current is defined to be the closed set
\begin{equation*}
  \supp(T) := \R^n\setminus \bg(\bigcup \B\{ U \subset \R^n\Bsep U\text{ is open and }\<T,\phi\>=0\text{ for all }\phi\in\mathscr{D}^m(U)\B\}\bg).
\end{equation*}
We recall that $\Sigma\subset\R^n$ is \emph{$m$--rectifiable} if it may be expressed as a countable union
of images of bounded subsets of $\R^m$ under Lipschitz maps, and an $m$--current is
\emph{rectifiable} if there exists an $m$--rectifiable
set $\Sigma\subset \R^n$, a Borel measurable function $\tau:\Sigma\to\Lambda_m\R^n$ with
$|\tau|=1$ on $\Sigma$, and a Borel measurable $\mu:\Sigma\to\N$ such that
\begin{equation*}
  \<T,\phi\> = \int_\Sigma \<\tau,\phi\>\mu\,\dd\Haus^m,
\end{equation*}
where $\Haus^m$ denotes the $m$--dimensional Hausdorff measure on $\R^n$.
This representation demonstrates that rectifiable $m$--currents generalise the elementary vector
calculus notion of integrals over $m$--dimensional subsets of $\R^n$.

An $m$--dimensional rectifiable current is an \emph{integral current} if both $T$ and $\partial T$
are rectifiable currents. We will denote the space of integral $m$--currents as
$\mathscr{I}_m$, and we exclusively consider currents in these classes.

\subsection{Pushforward and pullback}
\label{sec:pushforward-pullback}
Given a Lipschitz map $f:\R^m\to\R^n$, we recall that $f$ has a Frechet derivative 
$Df(x)\in\BLO(\R^m;\R^n)$ for almost every $x\in\R^m$,
where $\BLO(\R^m;\R^n)$ is the space of bounded linear operators mapping $\R^m$ to $\R^n$.

Following the definitions in \S4.3A of \cite{Morgan} (which in turn follow the definitions given in \S4.1.6 and \S4.1.7 of \cite{Federer}), 
we define the pushforward of a simple $m$--vector $v_1\wedge\dots\wedge v_m\in\Lambda_m\R^n$
under a linear map $A\in\BLO(\R^m,\R^n)$ to be
\begin{equation*}
  \Lambda_m(A)[v_1\wedge \dots\wedge v_m] := A[v_1] \wedge\dots\wedge A[v_m].
\end{equation*}
Noting that $\Lambda_m\R^n$ is spanned by simple $m$--vectors, this definition can be extended linearly to
apply to general $m$--vectors $\xi\in\Lambda_m\R^n$.
In particular, we recall that the key property
of the pushforward is that $\Lambda_m\b(Df(x)\b)$ gives the transformation of tangent space of a manifold under
the map $f$.

The \emph{pullback} of an $m$--form $\phi\in\mathscr{D}^m(\R^n)$ by
a Lipschitz map $f$, denoted $f^\#\phi$, is defined to be
\begin{equation*}
  \b\<\xi,f^\#\phi(x)\b\> = \b\<\Lambda_m\b(Df(x)\b)[\xi],\phi\b(f(x)\b)\b\>.
\end{equation*}
Note that $\phi$ is evaluated on
the target space $f(\R^m)\subseteq \R^n$.

This definition allows us to define the \emph{pushforward} of a rectifiable current by duality as
\begin{equation*}
  \<f_\#T,\phi\>:=\<T,f^\#\phi\>.
\end{equation*}

\bibliography{DDD}
\bibliographystyle{plain}

\end{document}